\documentclass[11pt,twoside,english,reqno,a4paper]{amsart}
\usepackage{listings,graphicx,amsmath,varioref,amscd,amssymb,color,bm,stmaryrd,amsthm,amsfonts,graphics,geometry,latexsym,pgf,pst-all}
\usepackage{tikz-cd}
\usetikzlibrary{patterns}

\theoremstyle{plain}
\usepackage{esint}
\usepackage{amsthm}
\theoremstyle{plain}
\newtheorem{theorem}{Theorem}[section]
\newtheorem{proposition}[theorem]{Proposition}

\newtheorem{lemma}[theorem]{Lemma}

\newtheorem{Corollary}[theorem]{Corollary}

\usepackage{geometry}
\usepackage[colorlinks=false]{hyperref}

\theoremstyle{definition}
\newtheorem{defin}[theorem]{Definition}

\newtheorem{remark}[theorem]{Remark}
\newtheorem{example}{Example}

\theoremstyle{remark}

\usepackage{fouriernc}
\usepackage[T1]{fontenc}

\def\bk{\color{black}}

\numberwithin{equation}{section}

\newcommand{\res}{\!\!\mathop{\hbox{
			\vrule height 7pt width .5pt depth 0pt
			\vrule height .5pt width 6pt depth 0pt}}
	\nolimits}
\newcommand{\DM }{\mathcal{DM}^\infty }

\newcommand{\Div}{\hbox{\rm div\,}}

\newcommand{\N}{\mathbb N}
\newcommand{\R}{\mathbb R}

\begin{document}

\title[Dirichlet problems for the 1--Laplacian with  $L^1$-data]{The Dirichlet problem for the $1$-Laplacian  with a general singular term and $L^1$-data}

\author[M. Latorre]{Marta Latorre}
\author[F. Oliva]{Francescantonio Oliva}
\author[F. Petitta]{Francesco Petitta}
\author[S. Segura de Le\'on]{Sergio Segura de Le\'on}

\address{}
\email{}

\address{Marta Latorre
\hfill\break\indent Matem\'atica Aplicada, Ciencia e Ingenier\'ia de los Materiales y Tecnolog\'ia Electr\'onica, Universidad Rey Juan Carlos
\hfill\break\indent C/Tulip\'an s/n 28933, M\'ostoles, Spain}\email{\tt marta.latorre@urjc.es}
\address{Francescantonio Oliva
\hfill\break\indent Dipartimento di Scienze di Base e Applicate per l' Ingegneria, Sapienza Universit\`a di Roma
\hfill\break\indent Via Scarpa 16, 00161 Roma, Italy}
\email{\tt francesco.oliva@sbai.uniroma1.it}
\address{Francesco Petitta
\hfill\break\indent Dipartimento di Scienze di Base e Applicate per l' Ingegneria, Sapienza Universit\`a di Roma
\hfill\break\indent Via Scarpa 16, 00161 Roma, Italy}
\email{\tt francesco.petitta@sbai.uniroma1.it}
\address{Sergio Segura de Le\'on
\hfill \break\indent Departament d'An\`alisi Matem\`atica, Universitat de Val\`encia
\hfill\break\indent Dr. Moliner 50,
46100 Burjassot, Val\`encia, Spain} \email{{\tt
sergio.segura@uv.es}}

%


\begin{abstract}
We study the Dirichlet problem for an elliptic  equation involving the $1$-Laplace operator and a reaction term, namely:
$$
\left\{\begin{array}{ll}
\displaystyle -\Delta_1 u =h(u)f(x)&\hbox{in }\Omega\,,\\
u=0&\hbox{on }\partial\Omega\,,
\end{array}\right.
$$
where  $\Omega\subset\R^N$ is an open bounded  set  having Lipschitz boundary, $f\in L^1(\Omega)$ is nonnegative,  and $h$ is a continuous real function that may possibly blow up at zero. We investigate optimal ranges for the data in  order to obtain existence, nonexistence  and (whenever expected)  uniqueness of nonnegative  solutions.
\end{abstract}

\maketitle

\tableofcontents

\section{Introduction}

This paper is concerned with the Dirichlet problem for an equation involving the $1$-Laplace operator and a lower order term in a bounded open set $\Omega\subset\R^N$ having Lipschitz boundary. More precisely, we deal with problem
\begin{equation}\label{P}
\left\{\begin{array}{ll}
\displaystyle -\Div\left(\frac{Du}{|Du|}\right)=h(u)f(x)&\hbox{in }\Omega\,,\\
u=0&\hbox{on }\partial\Omega\,,
\end{array}\right.
\end{equation}
where $f\in L^1(\Omega)$ and $h$ is a continuous real function; both  are assumed to be nonnegative and we look for nonnegative solutions. The function $h(s)$ is not necessarily required to be bounded near $s=0^+$.

Problems involving the $1$-Laplace operator are  known to be closely related to the mean curvature operator (\cite{OsSe}) and they  enter  in  a variety of practical   issues as (e.g. in the autonomous case $h\equiv 1$)  in image restoration and in  torsion problems (\cite{K, ka, Sapiro, M, BCRS}). The nonautonomous/nonsingular case in relation with more theoretical subjects such as eigenvalues problems and  critical Sobolev exponent has also been considered (see \cite{D, KS, Ch} and references therein).
We refer the interested reader to the monograph \cite{ACM} for a more complete review on applications in image processing.

For  what concerns our study, in particular,  it is worth recalling first the typical features of the  related autonomous problem
\begin{equation}\label{P0}
\displaystyle -\Div\left(\frac{Du}{|Du|}\right)=f(x)\quad\hbox{in }\Omega\,.
\end{equation}
This problem has been usually  addressed as the limit problem for $p$-laplacian type equations when $p$ goes to $1^+$.
In this setting it was considered by B. Kawohl in \cite{K} for constant data and by M. Cicalese and C. Trombetti in \cite{CT} for data belonging to the Marcinkiewicz space $L^{N,\infty}(\Omega)$. General data belonging to $W^{-1,\infty}(\Omega)$, the dual space of $W_0^{1,1}(\Omega)$, are taken in \cite{MST1}. Summarizing the results of these papers, it is proved that there exists a $BV$--function which is the limit of solutions to $p$--laplacian type equations only  when the datum has small norm.

Existence of solution to \eqref{P0} is also analyzed in \cite{MST1} taken into account the notion of solution for equations involving the 1--Laplace operator introducted in \cite{ABCM, D}. The $L^1$--setting is studied in \cite{MST2}, where it is observed that if $f\in L^1(\Omega)\backslash W^{-1, \infty}(\Omega)$, no almost everywhere finite solution can be expected (see \cite[Remarks 4.4 and 4.5]{MST2}). More precisely, it is proved that to get an almost everywhere finite solution, the datum $f$ must belong to $W^{-1,\infty}(\Omega)$ with $|| f||_{W^{-1,\infty}(\Omega)}\le 1$.

Concerning the presence of a non-constant term $h$ in  \eqref{P} one can refer to \cite{KS,D, Ch} for variational method in the study of eigenvalues type problems or critical Sobolev exponents. More recent papers analyzing the non autonomous case can be found in \cite{FP, MoSe}.
The case of a, possibly singular,  general nonlinearity $h$ in \eqref{P} has been first addressed in  \cite{dgs} in the case  $\displaystyle h(s)={s^{-\gamma}}$ (with $0<\gamma\le1$); here a $BV$--solution is found for every nonnegative $f\in L^N(\Omega)$. This existence result, without any smallness requirement, implies that the presence of $h(u)$ has a regularizing effect. Furthermore, in this paper it is shown the contrast between two cases of data: positive ($f>0$) and nonnegative ($f\ge0$). Roughly speaking, solutions of problem \eqref{P} with positive data have better properties and even a uniqueness result can be proved. In \cite{dgop} (see also \cite{O}) more general functions $h$ and data $f$ are analyzed and sharp existence and nonexistence results for \eqref{P} are given for nonnegative data in $f\in L^{N,\infty}(\Omega)$, completing the picture of \cite{MST1} and, as we already mentioned, emphasizing the regularity effect given by the nonlinear term in order to get non-trivial solutions.

\medskip
Here we want to  go further  focusing  on a different type of regularizing effect given by the nonlinear term $h$ if the data are assumed to be less regular. One of the purposes will in fact consist in    discriminate  whether  a blow-up phenomenon as the one described in \cite{MST2} does occur or not depending on the behavior of the nonlinearity at infinity.
Our main aim then will be the identification of  a general class of continuous functions $h$ such that there exists a solution to \eqref{P} for every nonnegative $f\in L^1(\Omega)$. We prove that it is so when $ h(s){\to}0$, as $s\to+\infty$. Moreover,  if $h$ attains the zero value, then we found a solution in $BV(\Omega)\cap L^\infty(\Omega)$, while if $h$ is separated from zero (that is: $h(s)\ge m>0$ for all $s>0$), then existence of a $BV$--solution is only possible when $f\in W^{-1,\infty}(\Omega)$. Therefore, the closer to zero is $h$, the greater the regularizing effect. Also we address to the interplay between  the  behavior of $h$ at infinity and the one near the origin, showing how  this produces non-degenerate solutions  no matter of the regularity or the smallness of the data.
\medskip

The main difficulties in this type of analysis rely, on one hand, on well known issues related to the presence of the singular operator $\Delta_1$, as for instance the need to give sense to the singular quotient $|Du|^{-1}Du$,   the  invariance under monotone transformations that (e.g.  in the autonomous case)  leads to a structural non-uniqueness property.  Moreover, one needs to work in  the $BV$ framework  where weaker compactness  results imply further complications as the ones related to the traces of the limit functions.

On the other hand, the presence of the general, and possibly singular,  nonlinear term $h$ makes the situation even worse as no monotonicity arguments apply and nearly all a priori estimates are only local.
To get rid of these problems we first need to construct an  auxiliary function that in some sense  controls $h$ and then  to carefully estimate suitable truncations of the approximate solutions.

 In order  to pass to the limit then, one has to take care of the possibly singular behavior of the lower order term and this leads to develop an argument that allows to manage the set where the approximating solutions are small; this is done again  by mean of  suitably truncated auxiliary functions.

\medskip

This paper is organized as follows.  Section 2 is devoted to introduce some  preliminary tools. In Section 3 we state our main assumptions, our notion of solution to problem \eqref{P} and our main results on existence and uniqueness in the case of positive data. These results are proved in Section 4. In Section 5, we analyze the case of nonnegative data. The last section is concerned with further remarks and examples  also casting lights on  the optimality of the presented results.

\section{Preliminaries}
\subsection{Fixing some notation}

We denote by $\mathcal H^{N-1}(E)$ the $(N - 1)$-dimensional Hausdorff measure of a set $E$ while $|E|$ stands for its $N$-dimensional Lebesgue measure.
 The characteristic function of a set $E$ will be written as $\chi_{E}$. \\
For the entire paper $\Omega\subset\mathbb{R}^N$ ($N\ge 1$) is open and bounded with Lipschitz boundary.  Thus, an outward normal unit   vector $\nu(x)$ is defined for $\mathcal H^{N-1}$--almost every   $x\in\partial\Omega$.
 We will make use of the usual Lebesgue and Sobolev
 spaces, denoted by $L^q(\Omega)$  and $W_0^{1,p}(\Omega)$, respectively, and $C_c^1(\Omega)$ stands for the set of all functions with compact support which are continuously differentiable on $\Omega$.

 Positive functions, that is, those satisfying $f(x)>0$ for almost all $x\in\Omega$, will be denoted by $f>0$, while $f\ge0$ stands for functions such that $f(x)\ge0$ almost everywhere in $\Omega$.

In what follows, $\mathcal{M}(\Omega)$ is the usual space of Radon measures with finite total variation over $\Omega$. The space $\mathcal{M}_{\rm loc}(\Omega)$ is the space of Radon measures which are locally finite in $\Omega$, that is, measures with finite total variation on every $\omega\subset\!\subset\Omega$. Integrals with respect to Lebesgue measure will be denoted as $\int_\Omega f\, dx$, for a Lebesgue summable function $f$. For an integral with respect to a measure $\mu\in\mathcal{M}(\Omega)$ other than Lebesgue and Hausdorff ones we will use the notation $\int_\Omega f \mu$, instead of $\int_\Omega f \, d\mu$.

\medskip
For a fixed $k>0$, the truncation function $T_{k}:\mathbb{R}\mapsto\mathbb{R}$ is given by
\begin{align*}
T_k(s):=&\max (-k,\min (s,k))\,;
\end{align*}

while, for each $\delta>0$, we also use the following auxiliary function
\begin{align}\label{Vdelta}
\displaystyle
V_{\delta}(s):=
\begin{cases}
1  &s\le \delta\,, \\
\displaystyle\frac{2\delta-s}{\delta} \ \ &\delta <s< 2\delta\,, \\
0 \ \ &s\ge 2\delta\,.
\end{cases}
\end{align}

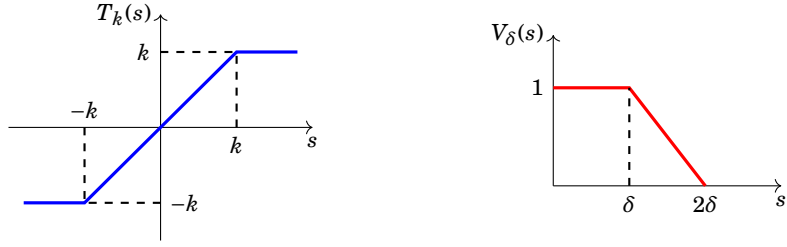
\begin{figure}[h]
\centering
\begin{minipage}{6cm}
\centering
\begin{tikzpicture}
\draw [->] (-2,0)--(2,0);
\draw [->] (0,-1.5)--(0,1.5);
\node [left] at (0,1.5) {{\footnotesize $T_k(s)$}};
\node [below] at (2,0) {{\footnotesize $s$}};

\node [below] at (1,0) {{\footnotesize $k$}};
\node [above] at (-1,0) {{\footnotesize $-k$}};
\node [left] at (0,1) {{\footnotesize $k$}};
\node [right] at (0,-1) {{\footnotesize $-k$}};
\draw [very thick, color=blue] (-1.8,-1)--(-1,-1)--(1,1)--(1.8,1);
\draw [thick, dashed] (-1,0)--(-1,-1)--(0,-1);
\draw [thick, dashed] (1,0)--(1,1)--(0,1);
\end{tikzpicture}
\end{minipage}
$ $
\begin{minipage}{6cm}
\centering
\begin{tikzpicture}
\draw [->] (0,0)--(3,0);
\draw [->] (0,0)--(0,2);
\node [left] at (0,2) {{\footnotesize $V_\delta(s)$}};
\node [below] at (3,0) {{\footnotesize $s$}};

\node [below] at (1,0) {{\footnotesize $\delta$}};
\node [below] at (2,0) {{\footnotesize $2\delta$}};
\node [left] at (0,1.3) {{\footnotesize $1$}};
\draw [very thick, color=red] (0,1.3)--(1,1.3)--(2,0);
\draw [thick, dashed] (1,0)--(1,1.3);
\end{tikzpicture}
\end{minipage}
\caption{Functions $T_k(s)$ and $V_\delta(s)$}
\label{Grafica_Tk_Vk}
\end{figure}

We explicitly remark that, if no otherwise specified, we will denote by $C$ several positive constants whose value may change from line to line and, sometimes, on the same line. These values will only depend on the data but they will never depend on the indexes of the sequences we will introduce.

\subsection{Functions of bounded variation}
The space of functions of bounded variation is defined by
$$BV(\Omega):=\{ u\in L^1(\Omega) : Du \in \mathcal{M}(\Omega)^N \}\,.$$
We also recall that every function $u\in BV (\Omega)$ has a trace defined on $\partial\Omega$ and $u\big|_{\partial\Omega}\in L^1(\partial\Omega)$.
We underline that the $BV(\Omega)$ space, endowed with the norm
$$ ||u||_{BV(\Omega)}=\int_\Omega |u|\,dx + \int_\Omega|Du|\,,$$
or with
$$\displaystyle ||u||_{BV(\Omega)}=\int_{\partial\Omega}
|u|\, d\mathcal H^{N-1}+ \int_\Omega|Du|\,,$$
is a Banach space and both norms are equivalent. We denote by $BV_{\rm loc}(\Omega)$ the space of functions in $BV(\omega)$ for every $\omega \subset\!\subset\Omega$. By $L_u$ we denote the set of approximate points of $u$, so that the approximate discontinuity set is $S_u =\Omega\setminus L_u$ and  with $J_u$ the set of the approximate jump points. It is well known that any function $u\in BV(\Omega)$ can be identified with its precise representative $u^*$, which is the Lebesgue representative $\tilde u$ in $L_u$, while $u^*=\frac{u^++u^-}{2}$ in $J_u$ where $u^+,u^-$ are the one--sided limits of $u$. Moreover, it can be shown that
$\mathcal{H}^{N-1}(S_u\setminus J_u)=0$ and that $u^*$ is $\mathcal{H}^{N-1}$--measurable. With an abuse of notation, in the sequel we will frequently write $g(u)^*$, where $g$ is merely continuous and $u$ belongs to the $BV$ space, meaning that $g(u)^*$ is equal to $g(\tilde{u})$ in $L_u$ while $g(u)^*=\frac{g(u^+) + g(u^-)}{2}$ in $J_u$.

When taken limits, we often apply the lower semicontinuity with respect to the $L^1$--convergence of the functionals given by
\begin{align*}
&  u\mapsto \int_\Omega\varphi|Du|\qquad\hbox{if }\varphi\in C_0^1(\Omega)\hbox{ and }\varphi\ge0\,,  \\
&   u\mapsto \int_\Omega|Du|+\int_{\partial\Omega}|u|\, \mathcal H^{N-1}\,.
\end{align*}

For a more complete review regarding $BV$ spaces we refer to \cite{AFP} from which we mainly derive our notation.

\subsection{$L^\infty$-divergence-measure vector fields}
We denote by
$$\DM(\Omega):=\{ z\in L^\infty(\Omega)^N : \operatorname{div}z \in \mathcal{M}(\Omega) \}\,,$$
and by $\DM_{\rm loc}(\Omega)$ its local version, namely: the space of bounded vector fields $z$ such that $\operatorname{div}z \in \mathcal{M}_{\rm loc}(\Omega)$.
For our purposes we recall the $L^\infty$-divergence-measure vector fields theory due to \cite{An} and \cite{CF}.
We first recall that if $z\in \DM(\Omega)$ then $\operatorname{div}z $ is a measure absolutely continuous with respect to $\mathcal H^{N-1}$.
\\In \cite{An} the following distribution $(z,Dv): C^1_c(\Omega)\mapsto \mathbb{R}$ is considered:
\begin{equation}\label{dist1}
\langle(z,Dv),\varphi\rangle:=-\int_\Omega v^*\varphi\operatorname{div}z-\int_\Omega
vz\cdot\nabla\varphi\, dx\,.
\end{equation}
In \cite{MST2} and \cite{C} the authors prove that $(z, Dv)$ is well defined if $z\in \DM(\Omega)$ and $v\in BV(\Omega)\cap L^\infty(\Omega)$ since one can show that $v^*\in L^\infty(\Omega,\operatorname{div}z)$. Moreover, in \cite{dgs} the authors show that \eqref{dist1} is well posed if $z\in \DM_{\rm loc}(\Omega)$ and $v\in BV_{\rm loc}(\Omega)\cap L^1_{\rm loc}(\Omega, \operatorname{div}z)$. In all cases, it holds that
\begin{equation*}\label{finitetotal}
|\langle   (z, Dv), \varphi\rangle| \le ||\varphi||_{L^{\infty}(U) } ||z||_{L^\infty(U)^N} \int_{U} |Dv|\,,
\end{equation*}
for all open set $U \subset\subset \Omega$ and for all $\varphi\in C_c^1(U)$. Moreover one has
\begin{equation}\label{finitetotal1}
\left| \int_B (z, Dv) \right|  \le  \int_B \left|(z, Dv)\right| \le  ||z||_{L^\infty(U)^N} \int_{B} |Dv|\,,
\end{equation}
for all Borel sets $B$ and for all open sets $U$ such that $B\subset U \subset \Omega$.\\
We recall that in \cite{An} it is proved that every $z \in \mathcal{DM}^{\infty}(\Omega)$ possesses a weak trace on $\partial \Omega$ of its normal component which is denoted by
$[z, \nu]$ (recall that $\nu(x)$ is the outward normal unit vector). Moreover, it holds
\begin{equation*}\label{des1}
||[z,\nu]||_{L^\infty(\partial\Omega)}\le ||z||_{L^\infty(\Omega)^N}\,,
\end{equation*}
and it also satisfies that if $z \in \mathcal{DM}^{\infty}(\Omega)$ and $v\in BV(\Omega)\cap L^\infty(\Omega)$, then
\begin{equation}\label{des2}
v[z,\nu]=[vz,\nu]\,,
\end{equation}
(see \cite{C}).\\
Finally, we will also use the Green formula due to \cite{dgs}. If $z\in \DM_{\rm loc}(\Omega)$ and $v\in BV(\Omega)\cap L^\infty(\Omega)$ is such that $v^*\in L^1(\Omega,\operatorname{div}z)$, then $vz\in \DM(\Omega)$ and a weak trace can be defined as well as the following Green formula holds.
\begin{lemma}
	Let $z\in \DM_{\rm loc}(\Omega)$ and let $v\in BV(\Omega)\cap L^\infty(\Omega)$ be such that $v^*\in L^1(\Omega,\operatorname{div}z)$. Then $vz\in \DM(\Omega)$ and it holds
	\begin{equation}\label{green}
	\int_{\Omega} v^* \operatorname{div}z + \int_{\Omega} (z, Dv) = \int_{\partial \Omega} [vz, \nu] \ d\mathcal H^{N-1}\,.
	\end{equation}	
\end{lemma}

\subsection{Construction of an auxiliary function}
\label{sec:h}

Our method of finding a solution to problem \eqref{P} is based on the possibility of taking $1/h(u)$ as test function (at least when $u$ is large enough). Obviously $h$ is merely a continuous function and we are not allowed to use it. Hence, we need to have a suitable approximation $\overline h$ of $h$ which may be used. This subsection is devoted to define this approximation.

Let $h:(0,\infty)\mapsto (0,\infty)$ be a continuous function such that there exists finite
$$\lim_{s\to \infty}  h(s) := {h}(\infty)\,.$$
Then, the aim is to construct a decreasing function  $\overline h\>:\>(0,\infty)\mapsto (0,\infty)$ satisfying
\begin{enumerate}
	\item[a)] $\overline h(s)\ge  h(s)$ for all $s\ge0$;
	\item[b)] $\overline h$ is locally Lipschitz--continuous;
	\item[c)] $\displaystyle \lim_{s\to \infty} \overline h(s):= {h}(\infty)$.
\end{enumerate}
We may define $\overline h$ through the following stages.
\begin{enumerate}
	\item[(i)] In the first step, we get a function $h_1(s)$ which is a $C^1$-approximation to
$h(s)+e^{-s}$ lying between  $h(s)$ and $h(s)+2e^{-s}$. To this end, we proceed as in the proof of
Meyers-Serrin's Theorem. Consider the intervals $(\frac1{n+2}, \frac1n)$ and $(n, n+2)$,
$n\in\N$, jointly with $(\frac12, 2)$. They form an open cover of $(0,+\infty)$. Then a smooth
partition of unity can be selected. Using mollifiers, choose a $C^1$-approximation to $h(s)
+e^{-s}$ between  $h(s)$ and $h(s)+2e^{-s}$ in each interval of the cover, and assemble
them by means of the partition of unity to obtain a function $h_1$. Thus, $h_1$ is a
$C^1$-function defined on $(0,+\infty)$ such that $h_1(s)\ge h(s)$ for all $s>0$.
	\item[(ii)] We denote by $h_2$ the rising sun function of $h_1$ (see, for instance, \cite[Exercises 1G and 4F]{RS}). We recall that the set of all shadow points of $h_1$ is an open set and can be written as a disjoint union $\cup_{j\in\N}(a_j, b_j)$. The procedure to pass from $h_1$ to $h_2$ can easily be translated to their derivatives: It just consists in changing the value of $h'_1$ on any $(a_j, b_j)$ by defining $h'_2(s)=0$. Observe that $h'_1$ is a continuous function and so it is bounded from below on any interval $(0,n]$, with $n\in\N$. We infer that $h'_2$ is non--positive and bounded from below on any interval $(0,n]$, with $n\in\N$, although it may be discontinuous at the points $a_j$. Hence, $h_2$ is a non--increasing locally Lipschitz--continuous function such that $h_2(s)\ge  h(s)$ for all $s>0$ and $\displaystyle \lim_{s\to \infty} h_2(s):= h(\infty)$.
	\item[(iii)] We define $\overline h(s)=h_2(s)+ e^{-s}$ which is decreasing, and so a), b) and c) hold.
\end{enumerate}

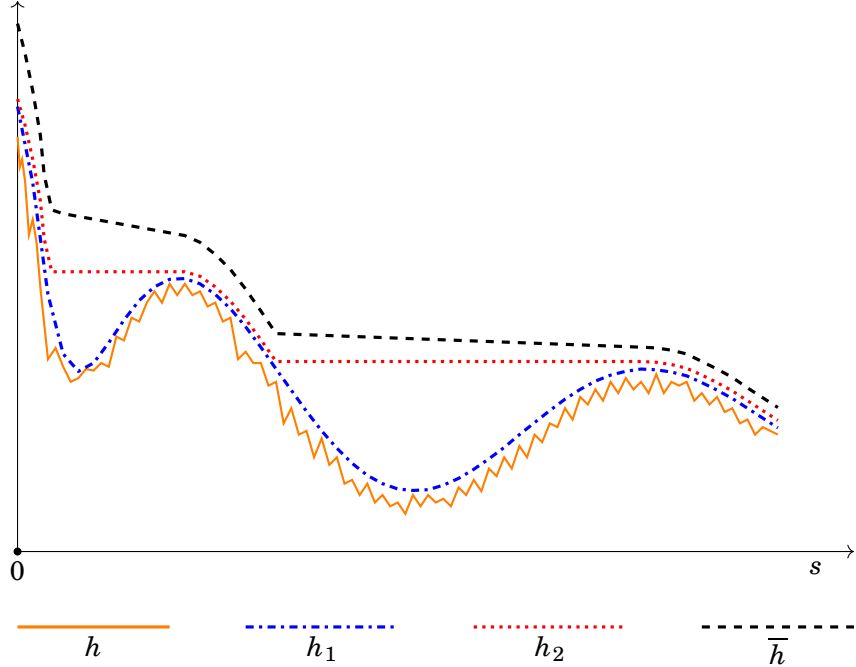
\begin{figure}
\centering
\begin{tikzpicture}
\draw[<->] (0,8.3)--(0,1)--(11,1);
\fill (10.5,1) circle (0.1pt) node[below] {$s$};
\fill (0,1) circle (1.5pt) node[below] {$0$};
\fill (1,0) node[below] {$h$};
\fill (4,0) node[below] {$h_1$};
\fill (7,0) node[below] {$h_2$};
\fill (10,0) node[below] {$\overline h$};
\draw [very thick, orange]  (0,0)--(2,0);
\draw [very thick, dash dot, blue] (3,0)--(5,0);
\draw [very thick, dotted, red]  (6,0)--(8,0);
\draw [very thick, dashed]  (9,0)--(11,0);

\draw [thick, orange]  (0,6.50)--(0.03,6.10)--(0.06,6.20)--(0.1,5.90)--(0.15,5.20)--(0.2,5.40)--(0.25,5.10)--(0.3,4.50)--(0.4,3.55)--(0.5,3.70)--(0.6,3.45)--(0.7,3.25)--(0.8,3.30)--(0.9,3.42)--(1,3.4)
--(1.1,3.50)--(1.2,3.45)--(1.3,3.85)--(1.4,3.80)--(1.5,4.10)--(1.6,4.05)--(1.7,4.30)--(1.8,4.45)--(1.9,4.30)--(2,4.55)
--(2.1,4.40)--(2.2,4.55)--(2.3,4.40)--(2.4,4.45)--(2.5,4.25)--(2.6,4.30)--(2.7,4.05)--(2.8,4.10)--(2.9,3.55)--(3,3.65)
--(3.1,3.50)--(3.2,3.50)--(3.3,3.20)--(3.4,3.25)--(3.5,2.70)--(3.6,2.90)--(3.7,2.55)--(3.8,2.60)--(3.9,2.25)--(4,2.50)
--(4.1,2.15)--(4.2,2.25)--(4.3,1.90)--(4.4,1.95)--(4.5,1.75)--(4.6,1.90)--(4.7,1.65)--(4.8,1.75)--(4.9,1.60)--(5,1.65)
--(5.1,1.50)--(5.2,1.75)--(5.3,1.60)--(5.4,1.75)--(5.5,1.65)--(5.6,1.70)--(5.7,1.60)--(5.8,1.85)--(5.9,1.75)--(6,1.95)
--(6.1,1.80)--(6.2,2.10)--(6.3,2.00)--(6.4,2.25)--(6.5,2.10)--(6.6,2.40)--(6.7,2.25)--(6.8,2.55)--(6.9,2.45)--(7,2.70)
--(7.1,2.65)--(7.2,2.90)--(7.3,2.75)--(7.4,3.05)--(7.5,2.95)--(7.6,3.15)--(7.7,3.00)--(7.8,3.25)--(7.9,3.10)--(8,3.30)
--(8.1,3.15)--(8.2,3.25)--(8.3,3.10)--(8.4,3.35)--(8.5,3.10)--(8.6,3.25)--(8.7,3.20)--(8.8,3.25)--(8.9,3.00)--(9,3.10)
--(9.1,2.95)--(9.2,3.05)--(9.3,2.85)--(9.4,2.90)--(9.5,2.70)--(9.6,2.75)--(9.7,2.55)--(9.8,2.65)--(9.9,2.60)--(10,2.55);
\draw [very thick, dash dot, blue]  (0,6.90)--(0.2,5.89)--(0.4,4.41)--(0.6,3.62)--(0.8,3.39)--(1,3.50)
--(1.2,3.77)--(1.4,4.07)--(1.6,4.33)--(1.8,4.51)--(2,4.61)
--(2.2,4.62)--(2.4,4.55)--(2.6,4.41)--(2.8,4.22)--(3,3.98)
--(3.2,3.71)--(3.4,3.42)--(3.6,3.14)--(3.8,2.86)--(4,2.60)
--(4.2,2.37)--(4.4,2.17)--(4.6,2.01)--(4.8,1.90)--(5,1.83)
--(5.2,1.81)--(5.4,1.82)--(5.6,1.88)--(5.8,1.97)--(6,2.09)
--(6.2,2.23)--(6.4,2.38)--(6.6,2.55)--(6.8,2.71)--(7,2.87)
--(7.2,3.02)--(7.4,3.15)--(7.6,3.26)--(7.8,3.34)--(8,3.39)
--(8.2,3.42)--(8.4,3.41)--(8.6,3.38)--(8.8,3.33)--(9,3.24)
--(9.2,3.14)--(9.4,3.03)--(9.6,2.90)--(9.8,2.77)--(10,2.64);
\draw [very thick, dotted, red] (0,7.00)--(0.1,6.65)--(0.2,6.20)--(0.3,5.65)--(0.35,5.15)--(0.45,4.71)--(0.6,4.71)
--(2.2,4.71)--(2.4,4.65)--(2.6,4.51)--(2.8,4.32)--(3,4.08)
--(3.2,3.81)--(3.4,3.52)
--(8.2,3.52)--(8.4,3.51)--(8.6,3.48)--(8.8,3.43)--(9,3.34)
--(9.2,3.24)--(9.4,3.13)--(9.6,3.00)--(9.8,2.87)--(10,2.74);
\draw [very thick, dashed]  (0,8.00)--(0.1,7.60)--(0.2,7.11)--(0.3,6.52)--(0.35,6.00)--(0.45,5.53)--(0.6,5.48)
--(2.2,5.19)--(2.4,5.10)--(2.6,4.94)--(2.8,4.74)--(3,4.48)
--(3.2,4.19)--(3.4,3.89)--(8.2,3.71)--(8.4,3.70)--(8.6,3.67)--(8.8,3.62)--(9,3.52)
--(9.2,3.42)--(9.4,3.31)--(9.6,3.17)--(9.8,3.04)--(10,2.91);
    \end{tikzpicture}
    \caption{The construction of $\overline{h}$}
    \label{fig:4}
    \end{figure}


\section{Assumptions and main results}
\label{main}
Consider the following Dirichlet problem
\begin{equation}
\label{pb}
\begin{cases}
\displaystyle  -\Delta_1 u = h(u)f(x) & \text{in}\;\Omega\,,\\
u\ge 0 & \text{in}\;\Omega\,,\\
u=0 & \text{on}\;\partial\Omega\,,
\end{cases}
\end{equation}
where the function $h:[0,\infty)\mapsto (0,\infty]$ is continuous, finite outside the origin, and  satisfying the following growth condition:
\begin{equation}\label{h1}\tag{h1}
\begin{aligned}
\displaystyle \exists\;{c},s_1>0, \gamma\ge 0 \;\ \text{such that}\;\  h(s)\le \frac{c}{s^\gamma} \ \ \text{if} \ \ s\leq s_1\,.
\end{aligned}
\end{equation}
We also assume that $h(0)\not= 0$ and  that there exists a finite limit at infinity of the function $h$, i.e.
\begin{equation}\label{h2}\tag{h2}
	 \displaystyle \lim_{s\to \infty} h(s):=h(\infty)\,.
\end{equation}
Here  we state the results relative to the case $h(\infty)=0$, while the   case $h(\infty)>0$ will be discussed  in Section \ref{hpositiva} below.

\medskip

The main feature of problem  \eqref{pb} is that the datum $f$ is merely integrable. Here we state the result in case that $f$ is positive, postponing the discussion of the nonnegative case to Section \ref{fnon}. Let us also underline that requiring that $h$ is positive is not restrictive at all: we refer to Section \ref{hzero} to handle the simpler case $h\ge 0$.
Finally we point out that the case of a  bounded $h$ is also covered by the above assumptions.

It is useful to fix the following notation:
$$\sigma:= \max(1,\gamma)\,,$$
which will be used henceforth. 
\begin{defin}\label{weakdef}
	A nonnegative function $u$ having $T_k(u)\in BV_{\rm{loc}}(\Omega)$ for any $k>0$ is a solution to problem \eqref{pb} if there exists $z\in \DM(\Omega)$ with $||z||_{L^\infty(\Omega)^N}\le 1$ such that $h(u)f\in L^{1}(\Omega)$  and
	\begin{align}
	&-\operatorname{div}z =h(u)f \ \ \ \ \ \ \ \ \ \ \ \text{as measures in } \Omega\,, \label{def_eqdistr}
	\\
	&(z,DT_k(u))=|DT_k(u)|  \ \ \ \ \text{as measures in } \Omega \text{ for any } k>0\,,\label{def_campo}
	\end{align}
	and one of the following conditions holds:
	\begin{equation}
	\lim_{\epsilon\to 0^+}  \fint_{\Omega\cap B(x,\epsilon)} T_k(u(y)) dy = 0 \ \ \text{for any } k>0 \ \ \text{or} \ \ \ [z,\nu] (x)= -1 \label{def_bordo}\ \ \ \text{for  $\mathcal{H}^{N-1}$-a.e. } x \in \partial\Omega\,.
	\end{equation}				
\end{defin}	

\begin{remark}
Let us observe that Definition \ref{weakdef} extends the ones in  \cite{dgs, dgop} the main difference being that, as no $BV$ solutions (nor bounded one) are expected in general (see Example \ref{esempio1} in Section \ref{examples}), then conditions \eqref{def_campo} and \eqref{def_bordo} have to be satisfied  by truncations of the  solution. Moreover it is  also  directly related  to  the definition given in \cite{MST2} in the {\it infinite energy} regime of the autonomous case.  A final comment on the weak way the boundary datum is assumed; if $T_k(u)$ does possesses a trace in the usual $BV$ sense (e.g. if $\gamma\leq 1$, as we shall see) then the Dirichlet boundary condition \eqref{def_bordo} reduces to the more easily readable one
$$
T_k (u) (1+[z,\nu] (x) )= 0\ \ \ \text{for  $\mathcal{H}^{N-1}$-a.e. } x \in \partial\Omega\,.
$$
\end{remark}

\medskip

We now state existence and uniqueness theorems in the case $f>0$:
\begin{theorem}\label{teomain}
	Let $0<f\in L^1(\Omega)$ and let $h$ satisfy \eqref{h1} and \eqref{h2} with $h(\infty)=0$. Then there exists a solution $u$ to problem \eqref{pb} in the sense of Definition \ref{weakdef} which is almost everywhere finite.
\end{theorem}

\begin{theorem}\label{teounique}
	Let $0<f\in L^1(\Omega)$ and let $h$ be decreasing. Then there exists at most one solution $u$ to \eqref{pb} in the sense of Definition \ref{weakdef} in the class $T_k^\sigma(u)\in BV(\Omega)$ for any $k>0$.
\end{theorem}

We also have the following  results that fit with  the known nonexistence ones and that highlight the regularization effect produced by the nonlinear term:
\begin{proposition}\label{propro}
	If $u$ is a solution to \eqref{pb}, then $h(u)f\in W^{-1,\infty}(\Omega)$ and $||h(u)f||_{W^{-1,\infty}(\Omega)}\le 1$.
\end{proposition}
\begin{proposition}\label{propreg}
	Let $u$ be the solution to \eqref{pb} found in Theorem \ref{teomain}. Then it holds that $T_k^\sigma(u)\in BV(\Omega)$ for any $k>0$. Moreover if $h(0)=\infty$ then $u>0$ almost everywhere in $\Omega$.
\end{proposition}

\section{Proof of the results of Section \ref{main}}

\subsection{Approximation scheme and a priori estimates}

For $p>1$, we consider the following scheme of approximation
\begin{equation}
\label{pbp}
\begin{cases}
\displaystyle -\Delta_p u_p  = h_p(u_p)f_p (x) & \text{in}\;\Omega\,,\\
u_p=0 & \text{on}\;\partial\Omega\,,
\end{cases}
\end{equation}
where $h_p(s):= T_{\frac{1}{p-1}}\left(h(s)\right)$ and $f_p= T_{\frac{1}{p-1}}\left(f\right)$. The existence of a nonnegative weak solution $u_p\in W^{1,p}_0(\Omega)\cap L^\infty(\Omega)$ is guaranteed by \cite{ll}. We highlight that, since we are taking $p\to 1^+$, without loss of generality we can assume that $1<p<2$. The a priori estimates on the approximate solutions  do not depend of the sign of the datum $f$ so that they shall be established  for general nonnegative $f$'s and they'll  also be used in Section \ref{fnon}.

\medskip

For a reason that will be clear in a moment,   we consider the function $\overline h$ constructed in Section \ref{sec:h} relative to  ${g}(s) = T_l(h(s))$  and $l= \displaystyle \sup_{s\in [s_1,\infty)} h(s)$ (i.e. we cut $h$ only near zero) and we set $s_2>\max\{1, s_1\}$ satisfying $\overline h(s_2)<s_1^{-\sigma}$; the function $\overline h$ is locally Lipschitz--continuous  and satisfies $\overline{h}(s)\ge h(s)$ for any $s\ge s_2$ and $\displaystyle \lim_{s\to \infty}\overline{h}(s)= \lim_{s\to \infty} h(s)$. In order to get a first basic estimate we need a test function that mimics $\frac{1}{h(u)}$, to this end we define
\begin{equation*}\label{phi}
\Phi(s) := \begin{cases}
s^{\sigma}, \ \ &s < s_1\,,\\
\Psi(s), \ \ &s_1\le s\le s_2\,,\\		
1/\overline{h}(s), \ \ &s>s_2\,,
\end{cases}
\end{equation*}
where $\Psi$ is a Lipschitz increasing function which makes $\Phi$ locally Lipschitz as well. Moreover we denote by
\begin{equation*}
\displaystyle
\Gamma_p(s) := \int_{0}^{s} \Phi'(t)^{\frac{1}{p}} \ dt\,.
\end{equation*}

We have the following easy but fundamental a priori estimate:
\begin{lemma}\label{lemmaphi}
Let $h$ satisfy \eqref{h1} and \eqref{h2} and let $f\in L^1(\Omega)$ be nonnegative. Let $u_p$ be a solution of \eqref{pbp} then
\begin{equation}\label{stimap}
||\Gamma_p(u_p)||_{W^{1,p}_0(\Omega)}\le C\,,
\end{equation}
for some constant $C$ which does not depend on $p$.
\end{lemma}
\begin{proof} Recalling that $\sigma\geq 1$,  one observes that $\Gamma_p(0) = 0$ and that $u_p$ has zero trace, so that $\Gamma_p(u_p)\in W_0^{1,p}(\Omega)$.

	Let us take $\Phi(u_p)$ as a test function in \eqref{pbp} and use the properties of $\overline h$,  obtaining
	$$
	\begin{aligned}
		\int_\Omega |\nabla \Gamma_p(u_p)|^p\, dx &= \int_{\Omega} |\nabla u_p|^p \Phi'(u_p)\, dx = \int_{\Omega} h_p(u_p) f_p \Phi(u_p)\, dx
		\\
		& \le c\int_{\{u_p\le s_1\}} f\, dx+ \left[\max_{s\in (s_1,s_2)}h(s)\right] \Phi(s_2) \int_{\{s_1< u_p< s_2 \}} f \, dx+ \int_{\{u_p\ge s_2\}} f\, dx \le C\,,
	\end{aligned}
	$$
	for some positive constant $C$ which does not depend on $p$.
\end{proof}

Here is the second {\it global} a priori estimate that involves the truncations of $u_p$:

\begin{lemma}\label{lemmapotenzaTk}
	Let $h$ satisfy \eqref{h1} and \eqref{h2}, $f\in L^1(\Omega)$ be nonnegative, and let us fix $k>0$. Then $u_p$ satisfies
	\begin{equation}\label{stimapTk}
	||T_k^\frac{\sigma-1+p}{p}(u_p)||_{W^{1,p}_0(\Omega)}\le C_k\,,
	\end{equation}
	for some constant $C_k$ which does not depend on $p$.
\end{lemma}
\begin{proof}
	Let us take $T_k^\sigma(u_p)$ as a test function in \eqref{pbp} yielding to
	$$
	\begin{aligned}
	\sigma\left(\frac{p}{\sigma-1+p}\right)^p\int_\Omega |\nabla T_k^\frac{\sigma-1+p}{p}(u_p)|^p\, dx &= \int_{\Omega} h_p(u_p) f_p T_k^\sigma(u_p)\, dx
	\\
	&\le c\int_{\{u_p\le s_1\}} f\, dx + \left[\sup_{s\in [s_1,\infty)}h(s)\right] \ k^\sigma\int_{\{u_p> s_1\}} f \, dx\le C_k\,,
	\end{aligned}
	$$	
	for some positive constant $C_k$ which does not depend on $p$.
\end{proof}

As we saw if $\gamma>1$ then only a power of the truncations of $u_p$ is bounded globally; this is a well known phenomenon  in singular type problems  (see e.g. \cite{BO,OP}). More in general we have the following local estimates

\begin{lemma}\label{lemmalocaleTk}
	Let $h$ satisfy \eqref{h1} and \eqref{h2} and let $f\in L^1(\Omega)$ be nonnegative.  Then for any $\omega \subset \subset \Omega$
	\begin{equation}\label{stimalocalep}
	||T_k(u_p)||_{W^{1,p}(\omega)}\le C_\omega\,,
	\end{equation}
	for some constant $C_\omega$ which does not depend on $p$.
\end{lemma}
\begin{proof}
	Let $\varphi\in C^1_c(\Omega)$ be a nonnegative function  with compact support in $\Omega$ such that $\varphi= 1$ in $\omega\subset\subset\Omega$; we take $(T_k(u_p)-k)\varphi^p$ as a test function in \eqref{pbp}, in order to get
	$$
	\begin{aligned}
	\int_\Omega |\nabla T_k(u_p)|^p\varphi^p\, dx + p\int_\Omega |\nabla T_k(u_p)|^{p-2} \nabla T_k(u_p)\cdot \nabla \varphi\ \varphi^{p-1} (T_k(u_p)-k)\, dx \le 0\,.
	\end{aligned}
	$$
	After  applying  twice  the Young inequality, we have
	  \begin{equation*}
	  \begin{aligned}
	\int_{\Omega} |\nabla T_k(u_p)|^{p}\varphi^p\, dx &\le  	(p-1)k\varepsilon \int_{\Omega} |\nabla T_k(u_p)|^{p}\varphi^{p}\, dx + k \varepsilon^{1-p}\int_{\Omega} |\nabla \varphi|^p\, dx
		 \\
	&
	\le  	k\varepsilon \int_{\Omega} |\nabla T_k(u_p)|^{p}\varphi^{p}\, dx + k\varepsilon^{1-p} \int_{\Omega} |\nabla \varphi|^2\, dx + k\varepsilon^{1-p}|\Omega|\,,
	\end{aligned}
	\end{equation*}
	and taking $\varepsilon$ sufficiently small  one easily concludes.
\end{proof}

\subsection{Estimates in BV and existence of an a.e. finite limit function $u$}

The estimates found in Lemmas \ref{lemmaphi},\ref{lemmapotenzaTk} and  \ref{lemmalocaleTk} translate into  uniform estimates in $BV$ with respect to $p$ that we collect in the following:

\begin{lemma}\label{stimeBV}
	Let $h$ satisfy \eqref{h1} and \eqref{h2} and let $f\in L^1(\Omega)$  nonnegative. Let $u_p$ be a solution of \eqref{pbp} then it holds that
	\begin{equation}\label{stimebv1}
	||\Gamma_p(u_p)||_{BV(\Omega)}\le C\,,
	\end{equation}
	\begin{equation}\label{stimebv2}
	||T_k^{\frac{\sigma-1+p}{p}}(u_p)||_{BV(\Omega)}\le C_k\,,
	\end{equation}
	and for any $\omega \subset \subset \Omega$
	\begin{equation}\label{stimebv3}
	||T_k(u_p)||_{BV(\omega)}\le C_\omega\,,
	\end{equation}	
	for some constants $C, C_k$ and $C_\omega$ which do not depend on $p$.
\end{lemma}
\begin{proof}
	The proofs of \eqref{stimebv1}, \eqref{stimebv2} and \eqref{stimebv3} are identical and rely on an application of the Young inequality. In fact, one has,
	$$\int_\Omega |\nabla \Gamma_p(u_p)|\, dx \le \int_\Omega |\nabla \Gamma_p(u_p)|^p\, dx + |\Omega| \overset{\eqref{stimap}}{\le} C +|\Omega|\,,$$
	which gives estimates \eqref{stimebv1} since the  trace of $u_p$ is zero and $\Gamma(0)=0$.
Similarly,
	$$\int_\Omega |\nabla T_k^\frac{\sigma-1+p}{p}(u_p)| \, dx\le \int_\Omega |\nabla T_k^\frac{\sigma-1+p}{p}(u_p)|^p \, dx+ |\Omega| \overset{\eqref{stimapTk}}{\le} C_k +|\Omega|\,,$$
	and, if $\omega\subset\subset\Omega$, one has
	$$\int_\omega |\nabla T_k(u_p)| \, dx\le \int_\omega |\nabla T_k(u_p)|^p\, dx + |\omega| \overset{\eqref{stimalocalep}}{\le} C_\omega +|\omega|\,.$$	
\end{proof}

We then deduce the following compactness results in which the most remarkable feature is given by the fact that the limit $u$ is finite a.e. on $\Omega$. As mentioned, this is due to the fact that $h$ vanishes at infinity and it is one of the striking differences with respect to the autonomous case considered in \cite{MST2}.

\begin{Corollary}\label{exu}
	Let $h$ satisfy \eqref{h1} and \eqref{h2} with $h(\infty)=0$ and let $f\in L^1(\Omega)$ be a  nonnegative function. Let $u_p$ be a solution of \eqref{pbp} then there exists a function $u$ almost everywhere finite such that, up to subsequences, $u_p$ converges to $u$ almost everywhere in $\Omega$ as $p\to 1^+$. Moreover, one has  that
	\begin{itemize}
		\item[i)] $\Gamma_p(u_p)$ converges to $\Phi(u)$ in $L^q(\Omega)$ with $1\le q<\frac{N}{N-1}$ and $\nabla \Gamma_p(u_p)$ converges *-weakly as measures to $D\Phi(u)$;			
		\item[ii)] $T_k^\frac{\sigma-1+p}{p}(u_p)$ converges to $T_k^\sigma(u)$ in $L^q(\Omega)$ with $1\le q<\frac{N}{N-1}$ and $\nabla T_k^\frac{\sigma-1+p}{p}(u_p)$ converges *-weakly as measures to $DT_k^\sigma(u)$ for any $k>0$;
		\item[iii)] $T_k(u_p)$ locally converges to $T_k(u)$ in $L^q(\Omega)$ with $1\le q<\frac{N}{N-1}$ and $\nabla T_k(u_p)$ converges locally *-weakly as measures to $DT_k(u)$ for any $k>0$.
	\end{itemize}
\end{Corollary}
\begin{proof}
	By compactness in $BV$ it follows from Lemma \ref{stimeBV}  that, up to subsequences, $v_p:=\Gamma_p(u_p)$ converges to a function $v\in BV(\Omega)$ a.e. and  in $L^q(\Omega)$ for every $1\le q<\frac{N}{N-1}$ and $\nabla v_p$ converges *-weakly as measures to $Dv$. Observe that if we define  $u := \Phi^{-1}(v)$ then  $u_p= \Gamma^{-1}_p(v_p)$  converges almost everywhere as $p \to 1^+$ to $u$. Finally, as $v = \Phi(u) \in L^q(\Omega)$ for all $1\le q<\frac{N}{N-1}$, one has  that $u$ is almost everywhere finite since $v$ is almost everywhere finite and $\Phi(s)$ blows up only as  $s\to \infty$. This shows that $u_p$ converges almost everywhere to $u$ as $p\to 1^+$ and it also proves i). The proofs of ii) and iii) follow in a similar  way. 	
\end{proof}

\subsection{Proof of the main existence result}

We will divide the proof in different lemmas, and then the proof of the existence will follow as the collection of these lemmas. First of all, let us underline that, from here on (even when not specified), $u_p$ will be a solution to \eqref{pbp} and $u$ is the almost everywhere limit as $p\to 1^+$ given by Corollary \ref{exu}.
  The first result concerns the existence of the vector field $z$ playing the role of the singular quotient $|Du|^{-1}Du$. In particular, the following lemma is the only place where the positivity of $f$ in $\Omega$ will be employed.
	\begin{lemma}\label{lemma_campoz}
	Let $h$ satisfy \eqref{h1}  and \eqref{h2} provided  $h(\infty)=0$,  and let $0< f\in L^1(\Omega)$. Let $u_p$ be a solution of \eqref{pbp} then there exists $z\in \mathcal{D}\mathcal{M}^\infty(\Omega)$ with $||z||_{L^\infty(\Omega)^N}\le 1$ such that $h(u)f \in L^1(\Omega)$ and such that
	\begin{equation}\label{lemma_distr}
	-\operatorname{div}z = h(u)f \text{  as measures in } \Omega\,.
	\end{equation}
	Moreover,
	\begin{equation}\label{lemma_identificarez}
	(z,DT_k(u))=|DT_k(u)| \ \ \ \ \text{as measures in } \Omega \text{ for any }k>0\,.
	\end{equation}		 	 	 		 	 	 	
\end{lemma}

\begin{proof}We begin by considering the set $\mathcal N$ of all levels $k>0$ such that $|\{u=k\}|>0$, which is a countable set. We point out that if $k\notin \mathcal N$, then
\begin{equation*}
  \chi_{\{u_p<k\}} \to \chi_{\{u<k\}}\qquad\hbox{strongly in } L^r(\Omega) \hbox{ for all }1\le r<\infty\,,
\end{equation*}
owing to the pointwise convergence $u_p\to u$ a.e. in $\Omega$.

For a fixed  $\omega\subset\subset\Omega$ one  obviously  has
\begin{equation}\label{conv2}
  \chi_{\omega\cap\{u_p<k\}} \to \chi_{\omega\cap\{u<k\}}\qquad\hbox{strongly in } L^r(\omega) \hbox{ for all }1\le r<\infty\,.
\end{equation}
	On the other hand, applying  \eqref{stimalocalep} and the H\"older inequality we deduce that for $1\le q<\frac{p}{p-1}$ and for any $k>0$,
	\begin{align}\label{stimaz}
	|||\nabla T_k(u_p)|^{p-2}\nabla T_k(u_p)||_{L^q(\omega)^N} \le \left(\int_\omega |\nabla T_k(u_p)|^p\, dx\right)^{\frac{p-1}{p}}|\omega|^{\frac{1}{q}-\frac{p-1}{p}}\leq C_\omega^{\frac{p-1}{p}}|\omega|^{\frac{1}{q}-\frac {p-1}{p}}\,,		
	\end{align}
that is, the family $\left(\big(|\nabla T_k(u_p)|^{p-2}\nabla T_k(u_p)\big)\chi_\omega\right)_p$ is bounded in $L^q(\omega)^N$ for every $1\le q<\infty$.
	This implies the existence of a vector field $z_{q,k}\in L^q(\omega)^N$ such that, up to subsequences,
$$\big(|\nabla T_k(u_p)|^{p-2}\nabla T_k(u_p)\big)\chi_\omega\rightharpoonup z_{q,k}\qquad\hbox{weakly  in }L^q(\omega)^N\,.$$
Now a standard  diagonal argument takes to the existence of a unique vector field $z_k$, defined independently of $q$, such that
\begin{equation}\label{conv1}
\left(|\nabla T_k(u_p)|^{p-2}\nabla T_k(u_p)\right)\chi_\omega \rightharpoonup z_k\quad\hbox{weakly  in }L^q(\omega)^N\hbox{ for any }1\le q<\infty\,.
\end{equation}
 Moreover, letting $p\to 1^+$, by weakly lower semicontinuity applied to \eqref{stimaz}, one yields to $||z_k||_{L^q(\omega)^N}\le |\omega|^{\frac1q}$ for any $q<\infty$ and letting $q\to \infty$ we also have $||z_k||_{L^\infty(\omega)^N}\le 1$.

The vector fields we have considered are defined in $\omega$: in order to explicit this dependence, we will write $z_k^\omega$ for $k\notin\mathcal N$. It remains to extend them to the whole $\Omega$. To this end, just take into account \eqref{conv1} to see that $\omega_1\subset\omega_2\subset\subset\Omega$ implies $z_k^{\omega_2}|_{\omega_1}=z_k^{\omega_1}$. Thus, it is enough to take an increasing sequence $\omega_n\subset\subset\Omega$ satisfying $\cup_{n=1}^\infty\omega_n=\Omega$ and so we obtain vector fields $z_k$ ($k\notin\mathcal N$) such that $z_k|_{\omega}=z_k^\omega$ and \eqref{conv1} holds for each $\omega\subset\subset\Omega$. It also follows from $||z_k^{\omega_n}||_{L^\infty(\omega_n)^N}\le 1$ for all $n\in\N$ that $||z_k||_{L^\infty(\Omega)^N}\le 1$.

Now, choose $h,k\notin \mathcal N$ satisfying $0<k<h$. Having in mind \eqref{conv1} and \eqref{conv2}, and letting $p$ go to $1$ in the identity
\[\left(|\nabla T_h(u_p)|^{p-2}\nabla T_h(u_p)\right)\chi_{\omega\cap\{u_p<k\}}=\left(|\nabla T_k(u_p)|^{p-2}\nabla T_k(u_p)\right)\chi_\omega\,,\]
we obtain
\begin{equation}\label{zk}z_h\chi_{\{u<k\}}=z_k\,.\end{equation}

As \eqref{zk} is in force, one can define  the vector field $z$ in $\Omega$ as
\[z(x)=z_k(x)\qquad \hbox{when }u(x)<k\quad (k\notin\mathcal N)\,.\]
It is straightforward that this vector field $z$ is well--defined in $\Omega$ (since $u$ is almost everywhere finite), $z\in L^\infty(\Omega)^N$ and $||z||_{L^\infty(\Omega)^N}\le 1$. Furthermore,
\[z\chi_{\{u<k\}}=z_k\qquad\hbox{ whatever }k\notin\mathcal N\,.\]

\medskip

	Now we look at \eqref{lemma_distr}; let us denote
	\begin{equation*}
		S_n(s) = \begin{cases}
		1 &\text{ if } s\le n\,, \\
		-s + n + 1 &\text{ if } n<s<n+1\,,\\
		0 &\text{ if } s\ge  n+1\,,
		\end{cases}
	\end{equation*}
and  let $0\le \varphi \in C^1_c(\Omega)$;  we take $S_n(u_p)\varphi$ as a test function in \eqref{pbp}, yielding to
    \begin{equation}\label{limite1}
    	\int_\Omega |\nabla u_p|^{p-2}\nabla u_p \cdot \nabla\varphi S_n(u_p)\, dx - \int_{\{n<u_p<n+1\}} |\nabla u_p|^{p}\varphi\, dx = \int_\Omega h_p(u_p)f_p S_n(u_p)\varphi\, dx\,,
    \end{equation}
    whence getting rid of the second term on the left hand side of \eqref{limite1} and taking $p\to1^+$ one has
 	\begin{equation*}
	\int_\Omega z \cdot \nabla\varphi S_n(u)\, dx \ge \int_\Omega h(u)f S_n(u)\varphi\, dx\,,
	\end{equation*}
	where, on the right hand side,  we have applied the Fatou Lemma. Again thanks  to Fatou's Lemma we    take $n\to \infty$,  getting
 	\begin{equation*}
	\int_\Omega z \cdot \nabla\varphi\, dx \ge \int_\Omega h(u)f \varphi\, dx\,,
	\end{equation*}	
	from which one deduces that   $z\in \DM_{\rm loc}(\Omega)$ and   $h(u)f \in L^1_{\rm loc}(\Omega)$. Let also observe that, in case $h(0)=\infty$, having $h(u)f\in L^1_{\rm loc}(\Omega)$ implies that
	\begin{equation}\label{uzero}
	\{u=0\}\subset\{f=0\},
	\end{equation}
	up to a set of zero Lebesgue measure; in particular, as $f>0$ one has that  $u>0$.
In order to prove that \eqref{lemma_distr} holds we now consider   $1-S_n(u_p)$ as a test function in \eqref{pbp} that yields
	$$\int_{\{n<u_p<n+1\}} |\nabla u_p|^p \, dx= \int_\Omega h_p(u_p)f_p (1-S_n(u_p))\, dx\,,$$
	and taking the limit as $p\to 1^+$ by means of the dominated convergence Theorem  (observe that $u_p\ge n$), one gets
	$$\displaystyle \lim_{p\to 1^+} \int_{\{n<u_p<n+1\}} |\nabla u_p|^p\, dx = \int_\Omega h(u)f(1-S_n(u))\, dx\,,$$
	and as $n\to \infty$, we have
	\begin{equation}\label{condrin}
	\displaystyle \lim_{n\to \infty}\lim_{p\to 1^+} \int_{\{n<u_p<n+1\}} |\nabla u_p|^p\, dx = 0\,,
	\end{equation}
	where in the last step we used again dominated convergence since $$h(u)f\chi_{\{n<u<n+1\}}\leq \sup_{s\geq 1} h(s)f \in L^1(\Omega)\,. $$
	Now we want to pass to the limit every terms of  \eqref{limite1} first in $p$ and then in $n$. We easily  pass to the limit the first term as well as the second term which goes to zero thanks to \eqref{condrin}.

	Concerning the  term on the right hand side   the limit is easy  if $h(0)<\infty$. Hence assume $h(0)=\infty$.
	 Let $\delta>0$ such that $\delta \notin\mathcal N$,  and write
	\begin{equation}\label{limite2}
		\int_\Omega h_p(u_p)f_pS_n(u_p)\varphi\, dx= \int_{\{u_p< \delta\}}h_p(u_p)f_pS_n(u_p)\varphi\, dx + \int_{\{u_p \ge\delta\}} h_p(u_p)f_pS_n(u_p)\varphi\, dx\,;
	\end{equation}
	 we want to pass to the limit in the following  order:  $p$,  then $n$,  and finally $\delta$ ($\to 0^+$). We first pass to the limit the second term on the right hand side of \eqref{limite2}; we observe that
	$$h_p(u_p)f_pS_n(u_p)\varphi\chi_{\{u_p>\delta\}}\le \sup_{s\in[\delta,\infty)}h(s) \ f\varphi \in L^1(\Omega)\,,$$
	which allows us to apply  Lebesgue's Theorem as $p\to 1^+$. One gains that
	$$ \displaystyle \lim_{p\to 1^+}\int_{\{u_p \ge \delta\}}h_p(u_p)f_pS_n(u_p)\varphi\, dx = \int_{\{u \ge \delta\}}h(u)fS_n(u)\varphi\, dx\,,$$
	and observing that $h(u)f\in L^1_{\rm loc}(\Omega)$, then one can apply once again the Lebesgue Theorem first in $n$ and then in $\delta$. Hence,  one has
	\begin{equation} \label{limite3}
	\displaystyle \lim_{\delta\to 0^+} \lim_{n\to \infty} \lim_{p\to 1^+}\int_{\{u_p \ge \delta\}}h_p(u_p)f_pS_n(u_p)\varphi\, dx = \int_	{\{u > 0\}}h(u)f\varphi\, dx\,.
	\end{equation}
Therefore, we are just left to show that the first term on the right hand side of \eqref{limite2} vanishes. To this aim we take $V_\delta(u_p)\varphi$ ($0\le \varphi \in C^1_c(\Omega)$ and $V_\delta(s)$ is defined in \eqref{Vdelta}) as a test function in \eqref{pbp} yielding to
	\begin{multline*}
		\int_{\{u_p< \delta\}} h_p(u_p)f_p\varphi\, dx \le \int_{\Omega} h_p(u_p)f_pV_\delta(u_p)\varphi\, dx \\
= \int_{\Omega} |\nabla u_p|^{p-2}\nabla u_p \cdot \nabla \varphi V_\delta(u_p) \, dx- \frac{1}{\delta}\int_{\{\delta<u_p<2\delta\}} |\nabla u_p|^p\varphi\, dx\,,
	\end{multline*}
	which gives
	\begin{equation}\label{limite4}
	\int_{\{u_p< \delta\}} h_p(u_p)f_p\varphi\, dx \le \int_{\Omega} |\nabla u_p|^{p-2}\nabla u_p \cdot \nabla \varphi V_\delta(u_p)\, dx\,.
	\end{equation}	
	Hence we can take $p\to 1^+$ in \eqref{limite4} obtaining
	\begin{equation*}
	\limsup_{p\to 1^+} \int_{\{u_p< \delta\}} h_p(u_p)f_p\varphi\, dx \le \int_{\Omega} z \cdot \nabla \varphi V_\delta(u)\, dx\,;
	\end{equation*}	
as there is no dependence on $n$,  we let $\delta \to 0^+$ obtaining
	\begin{equation}\label{limite5}
	\lim_{\delta\to 0^+}\limsup_{p\to 1^+} \int_{\{u_p< \delta\}} h_p(u_p)f_p\varphi\, dx \le \int_{\{u=0\}} z \cdot \nabla \varphi\, dx \overset{\eqref{uzero}}{=} 0\,.
	\end{equation}
	Summing up  (recall that $S_n(s)\le 1$)  \eqref{limite5} implies that
	\begin{equation*}\label{limite6}
		\lim_{\delta \to 0^+}\lim_{n\to \infty}\lim_{p\to 1^+}\int_{\{u_p< \delta\}}h_p(u_p)f_pS_n(u_p)\varphi\, dx =0\,,
	\end{equation*}
	which jointly with \eqref{limite3} gives
	\begin{equation*}\label{lebesgue}
	\lim_{n\to \infty}\lim_{p\to 1^+} \int_\Omega h_p(u_p)f_pS_n(u_p)\varphi\, dx= \int_{\{u > 0\}}h(u)f\varphi \, dx\overset{\eqref{uzero}}{=} \int_{\Omega}h(u)f\varphi\, dx.
	\end{equation*}	
	This proves that \eqref{lemma_distr} holds. Let us also underline that the validity of \eqref{lemma_distr} allows us to apply Lemma $5.3$ of \cite{dgop} in order to deduce that $\operatorname{div}z \in L^1(\Omega)$, i.e.  both $z\in \DM(\Omega)$ and $h(u)f \in L^1(\Omega)$.

In order to prove \eqref{lemma_identificarez} we first need to show that a slightly modified version of \eqref{lemma_distr} holds. For $k>0$
	we consider $(\rho_\epsilon\ast T_k^\sigma(u))\varphi$  as a test function in \eqref{lemma_distr} where $\rho_\epsilon$ is a sequence of standard mollifiers and $\varphi\in C^1_c(\Omega)$ is nonnegative. Hence
	\begin{equation*}\label{eqsigma}
	-\int_\Omega (\rho_\epsilon\ast T_k^\sigma(u)) \varphi\, \operatorname{div}z =  \int_\Omega h(u)f(\rho_\epsilon\ast T_k^\sigma(u))\varphi\, dx\,.
	\end{equation*}
	Observe that $T_k^\sigma(u)\in BV(\Omega)$, that $(\rho_\epsilon\ast T_k^\sigma(u)) \to T_k^\sigma(u)^*$ $\mathcal{H}^{N-1}-$almost everywhere as $\epsilon \to 0^+$,  and that $\rho_\epsilon\ast T_k^\sigma(u)\le k^\sigma$. Then letting $\epsilon \to 0^+$ one has
	\begin{equation}\label{eqsigma2}
- T_k^\sigma(u)^* \operatorname{div}z =   h(u)fT_k^\sigma(u) \ \ \text{as measures in }\Omega\,.
\end{equation}
	Now we take $T_k^\sigma(u_p)S_n(u_p)\varphi$ ($0\leq \varphi\in C^1_c(\Omega)$ and $n>k$) as a test function in \eqref{pbp}, deducing, after an application of the Young inequality, that
	\begin{equation}\label{inec:limit}
	\begin{aligned}	
	&\int_{\Omega} |\nabla T_k^{\frac{\sigma-1+p}{p}}(u_p)|\varphi\, dx + \int_{\Omega} T_k^\sigma(u_p)S_n(u_p)|\nabla u_p|^{p-2}\nabla u_p \cdot \nabla \varphi\, dx - k^\sigma\int_{\{n<u_p<n+1\}} |\nabla u_p|^{p}\varphi\, dx
	\\
	&\le \int_{\Omega}  h_p(u_p)f_p T_k^\sigma(u_p)S_n(u_p)\varphi\, dx + \frac{p-1}{p}\left(\frac{p}{\sigma-1+p}\right)^{\frac{p^2}{p-1}}\int_\Omega \varphi\, dx\,,	
	\end{aligned}
	\end{equation}
	and we want to  let  $p\to 1^+$ first, and then $n\to \infty$ in the \eqref{inec:limit}. We observe that by \eqref{condrin} the third term on the left hand side of \eqref{inec:limit} goes to zero as $p\to 1^+$ and $n\to\infty$. Then by lower semicontinuity and weak convergence at the left hand side of \eqref{inec:limit}, yielding to
	\begin{equation*}
	\begin{aligned}	
	&\int_{\Omega} \varphi|D T_k^\sigma(u)| + \int_{\Omega} T_k^\sigma(u)S_n(u)z \cdot \nabla \varphi\, dx - \displaystyle \lim_{p\to 1^+}k^\sigma\int_{\{n<u_p<n+1\}} |\nabla u_p|^{p}\varphi\, dx
	\\
	&\le \int_{\Omega}  h(u)f T_k^\sigma(u)S_n(u)\varphi\, dx\,,	
	\end{aligned}
	\end{equation*}	
	where   in order to   pass to the limit on the first term of right hand side we used dominated convergence Theorem  as
	$$
	h_p(u_p)f_p T_k^\sigma(u_p)S_n(u_p)\varphi\leq cf\varphi\chi_{\{u_p\leq s_1\}}+ \max_{s\in[s_1,\infty)}h(s)k^\sigma f \varphi\chi_{\{u_p> s_1\}}\leq (c+\max_{s\in[s_1,\infty)}h(s)k^\sigma)f\varphi\,.
	$$

Finally one can  take $n\to \infty$, obtaining
	\begin{equation*}
\begin{aligned}	
&\int_{\Omega} \varphi|D T_k^\sigma(u)| + \int_{\Omega} T_k^\sigma(u)z \cdot \nabla \varphi\, dx \le \int_{\Omega}  h(u)f T_k^\sigma(u)\varphi \, dx \overset{\eqref{eqsigma2}}{=} -\int_\Omega T_k^\sigma(u)^*\varphi \operatorname{div}z\,.
\end{aligned}
\end{equation*}		
Recalling \eqref{dist1} the previous implies that
	\begin{equation*}\label{primoverso}
	\int_{\Omega} \varphi|D T_k^\sigma(u)| \le \int_{\Omega}\varphi (z, D T_k^\sigma(u)),  \ \ \ \forall \varphi\in C^1_c(\Omega), \ \ \varphi \ge 0\,,	
	\end{equation*}	
	and since the reverse inequality is trivial, one has that
	\begin{equation}\label{ztk}
		(z,D T_k^\sigma(u)) =  |DT_k^\sigma(u)| \ \ \text{as measures in }\Omega\,,
	\end{equation}
	for any $k>0$. Finally one can apply Proposition $4.5$ of \cite{CDC} which implies
	\begin{equation}\label{cdc}
	\theta(z,DT_k(u),x)=\theta(z,DT_k^\sigma(u),x)  \qquad \text{for }|DT_k(u)|\text{-a.e. }\ x\in\Omega\,,
	\end{equation}
	where \(\theta(z, DT_k(u), \cdot)\) is the Radon-Nikod\'ym derivative of the measure $(z, DT_k(u))$ with respect to $|DT_k(u)|$, while \(\theta(z, DT_k^\sigma(u), \cdot)\) denotes the Radon-Nikod\'ym derivative of $(z, DT_k^\sigma(u))$ with respect to $|DT_k^\sigma(u)|$. This, jointly with \eqref{ztk}, implies \eqref{lemma_identificarez}. This concludes the proof of the lemma.
\end{proof}

Now we show that the Dirichlet boundary condition holds.

\begin{lemma}\label{lemmabordo}
Let $h$ satisfy \eqref{h1} and \eqref{h2} with $h(\infty)=0$  and let $0<f\in L^1(\Omega)$. Let $u$ be the function found in Lemma \ref{exu} and let $z$ be the vector field of  Lemma \ref{lemma_campoz},  then it holds 	 	 	
	\begin{align}\label{bordo}
	& \lim_{\epsilon\to 0^+} \fint_{\Omega\cap B(x,\epsilon)} T_k(u (y)) dy = 0 \ \ \ \text{or} \ \ \ [z,\nu] (x)= -1 \ \ \ \text{for  $\mathcal{H}^{N-1}$-a.e. } x \in \partial\Omega\,,
	\end{align}
	for any $k>0$.
\end{lemma}	
 	\begin{proof}
 	First of all we take $T_k^\sigma(u_p)$ as a test function in \eqref{pbp} and we apply the Young inequality in order to obtain (recalling  that the  trace of $u_p$ is zero)
	\begin{equation*}
	\int_{\Omega} |\nabla T_k^{\frac{\sigma-1+p}{p}}(u_p)|\, dx + \int_{\partial \Omega}T_k^\frac{\sigma-1+p}{p}(u_p) d\mathcal{H}^{N-1} \le  \int_{\Omega}  h_p(u_{p})f_p T_k^{\sigma}(u_{p})\, dx + \frac{p-1}{p}\left(\frac{p}{\sigma-1+p}\right)^{\frac{p^{2}}{p-1}}|\Omega|\,.	
	\end{equation*}		 	 	 	
	By lower semicontinuity on the left hand side and  by Lebesgue's dominated convergence Theorem  on the right hand side, letting $p\to 1^+$,  one has
	\begin{equation*}
	\int_{\Omega} |D T_k^\sigma(u)| + \int_{\partial \Omega}T_k^\sigma(u) d\mathcal{H}^{N-1} \le  \int_{\Omega}  h(u)f T_k^\sigma(u) \, dx \overset{\eqref{eqsigma2}}{=} -\int_{\Omega}(T_k^{\sigma}(u))^*\operatorname{div}z\,.
	\end{equation*}	 		 	 	 	
	Now an application of Green's identity \eqref{green} takes to

	\begin{equation}\label{bordo1}	\begin{aligned}
	\int_{\Omega} |D T_k^\sigma(u)| + \int_{\partial \Omega}T_k^\sigma(u) d\mathcal{H}^{N-1} &\le  \int_{\Omega}(z,D T_k^\sigma(u)) - \int_{\partial \Omega} [T_k^\sigma(u) z,\nu]d\mathcal{H}^{N-1}
	\\
	&= \int_{\Omega}|D T_k^\sigma(u)| - \int_{\partial \Omega} T_k^\sigma(u)[z,\nu]d\mathcal{H}^{N-1}\,,\end{aligned}
	\end{equation}
	where we have used \eqref{ztk} and the fact that $z\in \DM(\Omega)$. In particular \eqref{bordo1} implies that
	\begin{equation*}\label{eqbordo}T_k^\sigma(u)( 1 + [z,\nu]) = 0  \ \ \ \text{   $\mathcal{H}^{N-1}$- a.e. on $\partial \Omega$},\end{equation*}
	which means $[z,\nu](x)=-1$ or $T_k^\sigma(u(x))=0$ for $\mathcal{H}^{N-1}$-a.e. $x\in \partial\Omega$.
	Now, let $x\in \partial \Omega$ such that $T_k^{\sigma}(u(x))=0$,  then 		 $$\lim_{\epsilon\to 0^+} \fint_{\Omega\cap B(x,\epsilon)} T_k^{\sigma}(u (y)) dy=0\,.$$
	If $\sigma>1$,  by means of the H\"older inequality one gets,
	$$ \displaystyle
	\begin{array}{l}
	\displaystyle \fint_{\Omega\cap B(x,\epsilon)} T_k(u(y)) dy \leq \left( \fint_{\Omega\cap B(x,\epsilon)} T_k^{\sigma}(u(y)) dy\right)^{\frac{1}{\sigma}}\stackrel{\epsilon\to 0^+}{\longrightarrow} 0\,,
	\end{array}
	$$
	which concludes the proof.
\end{proof}

\begin{proof}[Proof of Theorem \ref{teomain}]
	Let $u_p$ be a solution to \eqref{pbp}. Then by Corollary \ref{exu} there exists $u$ finite almost everywhere such that $u_p$ converges almost everywhere to $u$ as $p \to 1^+$ and such that $T_k(u)\in BV_{\rm loc}(\Omega)$ for any $k>0$. Applying now Lemma \ref{lemma_campoz} there exists a vector field $z\in \DM(\Omega)$ with  $||z||_{L^\infty(\Omega)^N}\le 1$ such that \eqref{def_eqdistr} and \eqref{def_campo} hold. The same lemma gives that $h(u)f$ is integrable. Moreover condition \eqref{def_bordo} is proved in  Lemma \ref{lemmabordo}, and this concludes the proof of Theorem \ref{teomain}.
\end{proof}	

\begin{proof}[Proof of Proposition \ref{propro}]
Let $u$ be a solution to \eqref{pb}, so that there exists $z\in L^\infty(\Omega)^N$ such that $||z||_{L^\infty(\Omega)^N}\le1$ and
\begin{equation}\label{ast0}
  -\Div z=h(u)f\qquad\hbox{in }\mathcal D'(\Omega)\,.
\end{equation}
Consider $v\in W_0^{1,1}(\Omega)$, multiply   \eqref{ast0} by  $|T_k(v)|$  and apply Green's formula, then
\[\int_\Omega h(u)f |T_k(v)|\, dx=\int_\Omega z\cdot \nabla |T_k(v)|\, dx\le||z||_{L^\infty(\Omega)^N} \int_\Omega \big|\nabla |T_k(v)|\big|\, dx\le \int_\Omega|\nabla v|\, dx\,.\]
Letting $k$ go to infinity, we get
\[\left|\int_\Omega h(u)fv\, dx\right|\le \int_\Omega h(u)f|v|\, dx\le \int_\Omega|\nabla v|\, dx\,.\]
Since it holds for every $v\in W_0^{1,1}(\Omega)$, it follows that $h(u)f\in W^{-1,\infty}(\Omega)$ and its norm is less than 1.
\end{proof}	

\begin{proof}[Proof of Proposition \ref{propreg}]
The solution to \eqref{pb} given by Theorem \ref{teomain} is constructed through $u_p$ solutions to \eqref{pb}. Hence, having $T_k^\sigma(u)\in BV(\Omega)$ for any $k>0$ is consequence of Lemma \ref{stimeBV}. Finally, recalling that $h(u)f\in L^1_{\rm loc}(\Omega)$, we infer that if $h(0)=\infty$, then $u>0$  in $\Omega$ since $f>0$  in $\Omega$.
\end{proof}

\subsection{Proof of the uniqueness result}
In this subsection we prove the uniqueness result stated in Theorem \ref{teounique}.
\begin{proof}[Proof of Theorem \ref{teounique}]
	
	Let $u$ be a solution to \eqref{pb}, as   $\operatorname{div}z \in L^1(\Omega)$  one easily gets from  \eqref{ast0}
	\begin{equation}\label{testestese}
	-\int_{\Omega}v \operatorname{div}z\, dx = \int_{\Omega} h(u)fv\, dx\,,
	\end{equation}	
	for all $v\in BV(\Omega)\cap L^\infty(\Omega)$.	Suppose that $u_1$ and $u_2$ are two solutions to \eqref{pb}. Their associated vector fields are, respectively, $z_1$ and $z_2$. At this point we take $v=T_k^\sigma(u_1)- T_k^\sigma(u_2)$ ($k>0$) in the difference among the variational  formulations \eqref{testestese} solved, resp.,  by $u_1$ and $u_2$.
	\\This takes us to
	\begin{equation*}
	\begin{aligned}
	&\int_\Omega (z_1, DT_k^\sigma(u_1)) - \int_\Omega(z_1, DT_k^\sigma(u_2)) -\int_\Omega(z_2, DT_k^\sigma(u_1)) + \int_\Omega(z_2, DT_k^\sigma(u_2))\\
	&- \int_{\partial\Omega}(T_k^\sigma(u_1)-T_k^\sigma(u_2))[z_1,\nu])\, d\mathcal H^{N-1} +\int_{\partial\Omega}(T_k^\sigma(u_1)-T_k^\sigma(u_2))[z_2,\nu])\, d\mathcal H^{N-1}
	\\
	&= \int_\Omega (h(u_1)- h(u_2))f(T_k^\sigma(u_1)-T_k^\sigma(u_2))\, dx\,,
	\end{aligned}
	\end{equation*}
	where we have exploited  \eqref{green}. Moreover, recalling \eqref{def_campo}, one can use again \eqref{cdc} deducing that
	\begin{equation*}
	(z_i,D T_k^\sigma(u_i)) =  |DT_k^\sigma(u_i)| \quad \text{as measures in }\Omega, \quad i=1,2\,,
	\end{equation*}
	and recalling \eqref{def_bordo}, one yields to
	\begin{equation*}\label{unip=1_1}
	\begin{aligned}
	&\int_\Omega |DT_k^\sigma(u_1)| - \int_\Omega(z_1, DT_k^\sigma(u_2)) -\int_\Omega(z_2, DT_k^\sigma(u_1)) + \int_\Omega |DT_k^\sigma(u_2)|\\
	&+  \int_{\partial\Omega}(T_k^\sigma(u_1) +T_k^\sigma(u_1) [z_2,\nu]  )\, d\mathcal H^{N-1} +\int_{\partial\Omega}(T_k^\sigma(u_2) + T_k^\sigma(u_2)[z_1,\nu] )\, d\mathcal H^{N-1}	
	\\
	&\le \int_\Omega (h(u_1)- h(u_2))f(T_k^\sigma(u_1)-T_k^\sigma(u_2))\, dx\le 0\,.
	\end{aligned}
	\end{equation*}
	From the previous inequalities and from the facts that $||z_i||_{L^\infty(\Omega)^N} \le 1$, \eqref{finitetotal1}, and that $[z_i,\nu] \in [-1,1]$ for $i=1,2$, one deduces that 				    	
	$$\displaystyle \int_\Omega (h(u_1)- h(u_2))f(T_k^\sigma(u_1)-T_k^\sigma(u_2))\, dx= 0\,,$$
	which, since $h$ is decreasing and $f>0$,  gives $T_k^\sigma(u_1)=T_k^\sigma(u_2)$ a.e. in $\Omega$ for every $k>0$. In particular $u_1=u_2$ a.e. in $\Omega$ and the proof is concluded.
\end{proof}

\section{The case of a nonnegative $f$}
\label{fnon}
In this section we treat the case of a purely nonnegative datum $f$. The main point here is that one is not able to   show the strict positivity of the solution as a consequence of \eqref{uzero}. As we will see in Section \ref{examples} below, this is not a merely technical issue but a structural one and it is closely related to a non-uniqueness phenomenon.

We define the following function
$$\psi(s) = \begin{cases}1 \ \ \ &\text{if } h(0)<\infty, \\ \chi_{\{s>0\}} \ \ \ &\text{if } h(0)=\infty. \end{cases}$$
Now we are ready to give a  meaning to a solution of  \eqref{pb} in the case of a nonnegative datum.
	\begin{defin}\label{weakdefnonnegative}
	A nonnegative function $u$ having $T_k(u)\in BV_{\rm{loc}}(\Omega)$, $T_k^\sigma(u)\in BV(\Omega)$ for any $k>0$ and $\chi_{\{u>0\}}\in BV_{\rm loc}(\Omega)$ is a solution to problem \eqref{pb} if $h(u)f\in L^{1}_{\rm{loc}}(\Omega)$, there exists  $z\in \DM_{\rm loc}(\Omega)$ such that $||z||_{L^\infty(\Omega)^N}\le 1$  and
	\begin{align}
	&-\psi^*(u)\operatorname{div}z =  h(u)f \ \ \ \ \ \text{as measures in } \Omega\,, \label{nondef_distr}
	\\
	&(z,DT_k(u))=|DT_k(u)| \label{nondef_campo} \ \ \ \ \text{as measures in } \Omega \text{ for any k>0}\,,
	\\
	&  T^\sigma_k(u(x)) + [T_k^\sigma(u) z,\nu] (x)= 0 \label{nondef_bordo}\ \ \ \text{for  $\mathcal{H}^{N-1}$-a.e. } x \in \partial\Omega \text{ and for any } k>0\,.
	\end{align}
\end{defin}	

\begin{remark}\label{rema}
It should be noted that Definition \ref{weakdefnonnegative} (needed  for purely nonnegative data $f$) is nothing but a direct extension of   Definition \ref{weakdef}. Indeed if $f>0$ in   $\Omega$ and $h(0)=\infty$ (the finite case being simpler)  the request  $h(u)f\in L^1_{\rm loc}(\Omega)$ gives that $u>0$ in $\Omega$ and condition \eqref{nondef_distr} becomes $-\operatorname{div}z = h(u)f$. This allows us to apply Lemma $5.3$ of \cite{dgop} in order to deduce that $z\in \DM (\Omega)$, and then, recalling \eqref{des2}, one has that
	$$[T_k^\sigma(u) z,\nu]= T_k^\sigma(u)[ z,\nu]\,,$$
	which implies that \eqref{nondef_bordo} becomes
	$$T^\sigma_k(u(x))(1 + [z,\nu] (x))= 0\ \ \ \text{for  $\mathcal{H}^{N-1}$-a.e. } x \in \partial\Omega \text{ and for any } k>0\,.$$
	This identity implies  \eqref{def_bordo}. 		 				
\end{remark}

We state and prove the existence theorem for this case:

\begin{theorem}\label{existencenonnegative}
	Let $0\le f\in L^{1}(\Omega)$ and let $h$ satisfy \eqref{h1} and \eqref{h2} with $h(\infty)=0$. Then there exists a solution $u$ to problem \eqref{pb} in the sense of Definition \ref{weakdefnonnegative}.
\end{theorem}
\begin{proof}
We only sketch the proof, underlining the main differences with respect to the one of Theorem \ref{teomain}. We consider $u_p$ solutions to \eqref{pbp} and we apply Corollary \ref{exu} in order to deduce the existence of a function $u$ (finite almost everywhere) such that, up to subsequences, $u_p$ converges almost everywhere to $u$ as $p\to 1^+$.  Moreover, one has $T_k^\sigma(u)\in BV(\Omega)$ and $T_k(u) \in BV_{\rm loc}(\Omega)$ for any $k>0$. The existence of a vector field  $z\in \mathcal{D}\mathcal{M}^\infty_{\rm{loc}}(\Omega)$ with $||z||_{L^\infty(\Omega)^N}\le 1$ (weak limit in $L^q(\Omega)^N$ of $(|\nabla u_p|^{p-2}\nabla u_p)$ for every $1\le q<\infty$) follows exactly as in case $f>0$. Furthermore, analogously to the proof of Lemma \ref{lemma_campoz}, it can be proven that $ h(u)f\in L^{1}_{\rm{loc}}(\Omega)$ and that \eqref{nondef_campo} holds. The proof of \eqref{nondef_bordo} follows as in Lemma \ref{lemmabordo} up to:

\begin{equation*}\label{bordo2}	\begin{aligned}
	\int_{\Omega} |D T_k^\sigma(u)| + \int_{\partial \Omega}T_k^\sigma(u) d\mathcal{H}^{N-1} &\le  \int_{\Omega}(z,D T_k^\sigma(u)) - \int_{\partial \Omega} [T_k^\sigma(u) z,\nu]d\mathcal{H}^{N-1}\\
	&\le  \int_{\Omega}|D T_k^\sigma(u)| - \int_{\partial \Omega} [T_k^\sigma(u) z,\nu]d\mathcal{H}^{N-1}\,,
	\end{aligned}
	\end{equation*}
	from which \eqref{nondef_bordo} follows.

\medskip

	Now, we need to show that \eqref{nondef_distr} holds and that $\chi_{\{u>0\}}\in BV_{\rm loc}(\Omega)$. If $h(0)<\infty$ the proof of \eqref{nondef_distr} is trivial; hence we assume that $h(0)=\infty$ and we test \eqref{pbp} with  $R_\delta(u_p)\varphi$ where  $R_\delta(s):= 1-V_\delta(s)$ ($V_\delta$ is defined in \eqref{Vdelta}) and $0\le\varphi\in C^1_c(\Omega)$. After an application of the Young inequality, one gains
	\begin{equation*}\label{minore1}
	\int_{\Omega} |\nabla R_\delta(u_p)|\varphi\, dx + \int_{\Omega}|\nabla u_p|^{p-2} \nabla u_p\cdot \nabla \varphi R_\delta(u_p) \, dx \le   \frac{p-1}{p}\int_{\Omega}  \varphi\, dx  + \int_{\Omega}h_p(u_p)f_pR_\delta(u_p)\varphi\, dx\,,
	\end{equation*}
	and we intend to  pass to the limit first as $p\to1^+$, and then as $\delta\to 0^+$.
	First of all we can pass by lower semicontinuity in the first term on the left hand;  the second term easily passes to the limit as well. Regarding  the right hand side we have that the first term vanishes while  the second term passes to the limit via dominated convergence (recall that we are integrating where $u_p\ge \delta$).
	Hence one has
	\begin{align*}
	\int_{\Omega} |D R_\delta(u)|\varphi\, dx + \int_{\Omega}z\cdot \nabla \varphi R_\delta(u)\, dx\le \int_{\Omega}h(u)fR_\delta(u)\varphi\, dx\,.
	\end{align*}
	Since $h(u)f\in L^1_{\rm loc}(\Omega)$ and $z\in L^\infty(\Omega)^N$ one can pass to the limit the previous inequality as $\delta \to 0^+$, obtaining
	\begin{equation*}\label{equival}
	\int_{\Omega} |D \chi_{\{u>0\}}|\varphi + \int_{\Omega}z\cdot \nabla \varphi \chi_{\{u>0\}}\, dx \le \int_{\Omega}h(u)f\chi_{\{u>0\}}\varphi\, dx = \int_{\Omega}h(u)f\varphi\, dx\,,
	\end{equation*} 					
	where the last equality holds since $h(u)f \in L^1_{\rm loc}(\Omega)$ implies that $\{u=0\}\subset\{f=0\}$. Moreover this inequality also provides that $\chi_{\{u>0\}}\in BV_{\rm loc}(\Omega)$.
	Therefore, it follows from \eqref{dist1} that one has
	\begin{align}\label{minore4bis}
	-\int_{\Omega}\chi^*_{\{u>0\}}\varphi\operatorname{div}z \le \int_{\Omega}h(u)f\varphi\, dx\,.
	\end{align}
	We just need to prove the reverse inequality to  \eqref{minore4bis}. From the weak formulation of \eqref{pbp}, taking a nonnegative $\varphi\in C_c^\infty(\Omega)$ and letting $p\to 1^+$ by the Fatou Lemma one has
	\begin{equation*}
		-\int_\Omega\varphi\operatorname{div}z\, dx=\int_\Omega z\cdot \nabla \varphi\, dx \ge \int_\Omega h(u)f\, dx\,.
	\end{equation*}
	In particular, we can take $\varphi = (\chi_{\{u>0\}}* \rho_\epsilon)\phi$ where $0\le \phi \in C^1_c(\Omega)$ and $\rho_\epsilon$ is again a sequence   of standard mollifiers. Thus one can pass to the limit in $\epsilon$, obtaining  	
	\begin{equation*}\label{maggiore}
	-\int_{\Omega}\phi\chi^*_{\{u>0\}}\operatorname{div}z \ge \int_{\Omega}h(u)f\chi_{\{u>0\}}\phi\, dx =  		\int_{\Omega}h(u)f\phi\, dx \ \ \ \forall\phi \in C^1_c(\Omega), \ \ \phi \ge 0\,,
	\end{equation*}
	which proves \eqref{nondef_distr}. This concludes the proof.
\end{proof}

\section{Remarks and examples}

\subsection{The case of a nonnegative $h$}
\label{hzero}
We want to show that in case in which $h$ touches zero then one can find  a bounded solution even if $f$ is merely integrable. Let us suppose that there exists $\tilde{s} >0$ such that $h(\tilde{s})=0$ and $h(s)>0$ for all $0<s<\tilde{s}$. We consider the approximation given by \eqref{pbp} where
$$h_p(s)= \begin{cases}
T_{\frac{1}{p-1}}(h(s)), \ &0\le s\le \tilde{s}\,,\\
0 \,,		&s>\tilde{s}\,.
\end{cases}
$$
and taking $(u_p - T_{\tilde{s}}(u_p))^+$ as a test function in \eqref{pbp} we deduce that
$$\int_{\Omega} |\nabla (u_p - T_{\tilde{s}}(u_p))^+|^p\, dx \le 0\,,$$
which implies that $u_p \le \tilde{s}$. Hence, reasoning as in the proof of Theorem \ref{teomain},  one can prove that there exists a solution $u$ to problem \eqref{pb} which belongs to $L^\infty(\Omega)$ and satisfies $||u||_{L^{\infty}(\Omega)}\le \tilde{s}$.

It is worth remarking that the solution to this problem may not be unique (see Example \ref{ejemplo-no-unicidad} below).

\subsection{The case $h(\infty)>0$}
\label{hpositiva}

Let us now focus instead on the case in which $h(s)>0$ and  $$\displaystyle h(\infty)=\lim_{s\to\infty}h(s)>0\,.$$
 Then, the function $h$ is bounded from below by  $m>0$.
Assuming that for some positive data  $f$ there exists a solution $u$ with vector field $z$, it holds
\begin{equation*}
-\Div z=h(u)f\ge m f \text{ in }\Omega\,.
\end{equation*}
We use now a test function $v\in W_0^{1,1}(\Omega)$ to obtain
\begin{equation*}
\int_\Omega |\nabla v|\,dx\ge m \int_\Omega f |v|\,dx\,,
\end{equation*}
which implies
\begin{equation*}
\left|\int_\Omega fv\,dx \right|\le \int_\Omega f|v|\,dx \le \frac{1}{m} \int_\Omega|\nabla v|\,dx\,,
\end{equation*}
and one may deduce that, if exists a solution, then  $f\in W^{-1,\infty}(\Omega)$ and $||f||_{W^{-1,\infty}(\Omega)}\le \frac{1}{m}$. Moreover, the closer to $0$ is function $h$, the more general conditions over $f$ in order to have a solution.

One may also guess that, suitably  weakening the definition,  one can get existence of a solution; the proof, as $h(s)\geq m>0$,  being a step-by-step replica of the one in \cite{MST2}. Only observe that in this case the {\it solution} $u$ may be, in general,  $+\infty$ on a set of positive measure as well as the vector field $z$ could be unbounded.

\subsection{Examples}\label{examples}

This section is devoted to show some explicit examples of solutions;  we look for radial solutions.  Let  $\Omega= B_R(0)$ be a  ball centered at the origin  of radius $R>0$. Let  $\displaystyle h(s)=s^{-\gamma}$, $\gamma> 0$, that is we focus on
\begin{equation}\label{prob-examples}
\left\{
\begin{array}{ll}
\displaystyle -\Div \left(\frac{Du}{|Du|}\right) =\dfrac{f(x)}{u^\gamma} & \mbox{ in }  B_R(0) \,, \\[3mm]
u(x)\ge 0 & \mbox{ in }  B_R(0) \,,\\[2mm]
u(x)=0 & \mbox{ on } \partial B_R(0) \,,
\end{array}\right.
\end{equation}
for various type of data $f$.
We look for non-increasing radial solutions, i.e., $u(x)=g(|x|)$ with $g'(s)\le0$ for $0\le s \le R$. Since solutions have to be  nonnegative, the boundary condition becomes either $u\big|_{\partial B_R(0)}=0$ or $[z,\nu]=-1$. We remark that if $g'(s)<0$ in a zone of positive measure, then $z(x)=\frac{Du}{|Du|}=-\frac{x}{|x|}$ and $-\Div z=\frac{N-1}{|x|}$.

In the first  example we construct a solution for a datum $f$ with  summability between 1 and $N$:

\begin{example}\rm\label{esempio1}
We start taking the positive datum $f(x)=\frac{N-1}{|x|^q}$ with $1<q<N$ in problem \eqref{prob-examples}:
\begin{equation*}
\left\{
\begin{array}{ll}
\displaystyle -\Div \left(\frac{Du}{|Du|}\right) =\dfrac{N-1}{|x|^q}\dfrac{1}{u^\gamma(x)} & \mbox{ in }  B_R(0) \,, \\[3mm]
u(x)\ge 0 & \mbox{ in }  B_R(0) \,,\\[2mm]
u(x)=0 & \mbox{ on } \partial B_R(0)\,.
\end{array}\right.
\end{equation*}
Assuming that solution $u$ is non-constant (which implies that
the vector field is given by $z=-x/|x|$) and taking $|x|=r$, then equation becomes
\begin{equation*}
\dfrac{N-1}{r} = - \Div z = \dfrac{N-1}{r^q}\dfrac{1}{g^\gamma (r)}\,,
\end{equation*}
and therefore, the unique solution to this problem is given by $u(x)=g(|x|)={|x|^{\frac{1-q}{\gamma}}}$. Observe that  $u$ is positive, it has truncation in   $ BV(B_R(0))$,  but is unbounded.  Moreover,  it holds the equality $(z(x),D T_k(u(x)))=|D T_k(u(x))|$, for any $k>0$,  as measures in $B_R(0)$ and the boundary condition $[z,\nu]=-1$.

Note also that for $q\to 1^+$ solutions formally tend to a bounded function as the ones found in
\cite{dgop}. Also notice that, if $\gamma=1$,  as $q\to N^-$  (i.e. the datum goes narrowly outside $L^1$) then $u$, in the limit,  comes out to $L^{1^*}_{}(B_R(0))$, yielding  a sort of {\it asymptotic optimality} to our result.

\medskip

Furthermore,  assume $R>1$ and observe that if we formally let $\gamma\to 0^+$ then the solutions tend to
\begin{equation*}
u_{0}(x)=\left\{\begin{array}{lcl}
+\infty & \text{ if } & 0<|x|<1\,,\\[3mm]
1 & \text{ if } & |x|=1\,,\\[3mm]
0 & \text{ if } & 1<|x|<R\,,
\end{array}\right.
\end{equation*}
recovering the blow-up/degenerate phenomenon in \cite[Theorem $3.1$]{MST1}.

\begin{figure}[h]
\centering
\begin{tikzpicture}
\fill[color=blue!10] (0,0) circle (2cm);
\fill[color=white] (0,0) circle (1cm);
\draw[thick,dotted, orange] (0.95,0)--(0.95,-0.31);
\draw[thick,dotted, orange] (0.95,0.07)--(0.95,0.31);
\draw[thick,dotted, orange] (0.9,0)--(0.9,-0.43);
\draw[thick,dotted, orange] (0.9,0.07)--(0.9,0.43);
\draw[thick,dotted, orange] (0.85,0)--(0.85,-0.53);
\draw[thick,dotted, orange] (0.85,0.07)--(0.85,0.53);
\draw[thick,dotted, orange] (0.8,0)--(0.8,-0.6);
\draw[thick,dotted, orange] (0.8,0.07)--(0.8,0.6);
\draw[thick,dotted, orange] (0.75,0.07)--(0.75,0.66);
\draw[thick,dotted, orange] (0.75,0)--(0.75,-0.66);
\draw[thick,dotted, orange] (0.7,0.07)--(0.7,0.71);
\draw[thick,dotted, orange] (0.7,0)--(0.7,-0.71);
\draw[thick,dotted, orange] (0.65,0.07)--(0.65,0.76);
\draw[thick,dotted, orange] (0.65,0)--(0.65,-0.76);
\draw[thick,dotted, orange] (0.6,0.07)--(0.6,0.8);
\draw[thick,dotted, orange] (0.6,0)--(0.6,-0.8);
\draw[thick,dotted, orange] (0.55,0.07)--(0.55,0.83);
\draw[thick,dotted, orange] (0.55,0)--(0.55,-0.83);
\draw[thick,dotted, orange] (0.5,0.07)--(0.5,0.87);
\draw[thick,dotted, orange] (0.5,0)--(0.5,-0.87);
\draw[thick,dotted, orange] (0.45,0.07)--(0.45,0.89);
\draw[thick,dotted, orange] (0.45,0)--(0.45,-0.89);
\draw[thick,dotted, orange] (0.4,0.07)--(0.4,0.92);
\draw[thick,dotted, orange] (0.4,0)--(0.4,-0.92);
\draw[thick,dotted, orange] (0.35,0.07)--(0.35,0.94);
\draw[thick,dotted, orange] (0.35,0)--(0.35,-0.94);
\draw[thick,dotted, orange] (0.3,0.07)--(0.3,0.95);
\draw[thick,dotted, orange] (0.3,0)--(0.3,-0.95);
\draw[thick,dotted, orange] (0.25,0.07)--(0.25,0.97);
\draw[thick,dotted, orange] (0.25,0)--(0.25,-0.97);
\draw[thick,dotted, orange] (0.2,0.07)--(0.2,0.98);
\draw[thick,dotted, orange] (0.2,0)--(0.2,-0.98);
\draw[thick,dotted, orange] (0.15,0.07)--(0.15,0.99);
\draw[thick,dotted, orange] (0.15,0)--(0.15,-0.99);
\draw[thick,dotted, orange] (0.1,0.07)--(0.1,0.99);
\draw[thick,dotted, orange] (0.1,0)--(0.1,-0.99);
\draw[thick,dotted, orange] (0.05,0.07)--(0.05,1);
\draw[thick,dotted, orange] (0.05,0)--(0.05,-1);
\draw[thick,dotted, orange] (-0.95,0)--(-0.95,-0.31);
\draw[thick,dotted, orange] (-0.95,0.07)--(-0.95,0.31);
\draw[thick,dotted, orange] (-0.9,0)--(-0.9,-0.43);
\draw[thick,dotted, orange] (-0.9,0.07)--(-0.9,0.43);
\draw[thick,dotted, orange] (-0.85,0)--(-0.85,-0.53);
\draw[thick,dotted, orange] (-0.85,0.07)--(-0.85,0.53);
\draw[thick,dotted, orange] (-0.8,0)--(-0.8,-0.6);
\draw[thick,dotted, orange] (-0.8,0.07)--(-0.8,0.6);
\draw[thick,dotted, orange] (-0.75,0.07)--(-0.75,0.66);
\draw[thick,dotted, orange] (-0.75,0)--(-0.75,-0.66);
\draw[thick,dotted, orange] (-0.7,0.07)--(-0.7,0.71);
\draw[thick,dotted, orange] (-0.7,0)--(-0.7,-0.71);
\draw[thick,dotted, orange] (-0.65,0.07)--(-0.65,0.76);
\draw[thick,dotted, orange] (-0.65,0)--(-0.65,-0.76);
\draw[thick,dotted, orange] (-0.6,0.07)--(-0.6,0.8);
\draw[thick,dotted, orange] (-0.6,0)--(-0.6,-0.8);
\draw[thick,dotted, orange] (-0.55,0.07)--(-0.55,0.83);
\draw[thick,dotted, orange] (-0.55,0)--(-0.55,-0.83);
\draw[thick,dotted, orange] (-0.5,0.07)--(-0.5,0.87);
\draw[thick,dotted, orange] (-0.5,0)--(-0.5,-0.87);
\draw[thick,dotted, orange] (-0.45,0.07)--(-0.45,0.89);
\draw[thick,dotted, orange] (-0.45,0)--(-0.45,-0.89);
\draw[thick,dotted, orange] (-0.4,0.07)--(-0.4,0.92);
\draw[thick,dotted, orange] (-0.4,0)--(-0.4,-0.92);
\draw[thick,dotted, orange] (-0.35,0.07)--(-0.35,0.94);
\draw[thick,dotted, orange] (-0.35,0)--(-0.35,-0.94);
\draw[thick,dotted, orange] (-0.3,0.07)--(-0.3,0.95);
\draw[thick,dotted, orange] (-0.3,0)--(-0.3,-0.95);
\draw[thick,dotted, orange] (-0.25,0.07)--(-0.25,0.97);
\draw[thick,dotted, orange] (-0.25,0)--(-0.25,-0.97);
\draw[thick,dotted, orange] (-0.2,0.07)--(-0.2,0.98);
\draw[thick,dotted, orange] (-0.2,0)--(-0.2,-0.98);
\draw[thick,dotted, orange] (-0.15,0.07)--(-0.15,0.99);
\draw[thick,dotted, orange] (-0.15,0)--(-0.15,-0.99);
\draw[thick,dotted, orange] (-0.1,0.07)--(-0.1,0.99);
\draw[thick,dotted, orange] (-0.1,0)--(-0.1,-0.99);
\draw[thick,dotted, orange] (-0.05,0.07)--(-0.05,1);
\draw[thick,dotted, orange] (-0.05,0)--(-0.05,-1);
\draw[thick,dotted, orange] (0,0.07)--(0,1);
\draw[thick,dotted, orange] (0,0)--(0,-1);
\draw[thick] (0,0) circle (1cm);
\draw (0,0) circle (2cm);
\draw [->][thick] (0,0)--(0.7,0.7);
\draw [->][thick] (0,0)--(1.9,0.5);
\node at (0.4,0.6) {{\footnotesize $1$}};
\node at (1.3,0.5) {{\footnotesize $R$}};
\node at (-1.5,2.2) {{\footnotesize $B_R(0)$}};
\node at (-0.1,-0.4) {{\footnotesize $u_0=+\infty$}};
\node at (-1,0) {{\footnotesize $u_0=1$}};
\node at (0,1.5) {{\footnotesize $u_0=0$}};
\end{tikzpicture}
\caption{The limit $u_0(x)$ as $\gamma\to0^+$.}
\end{figure}
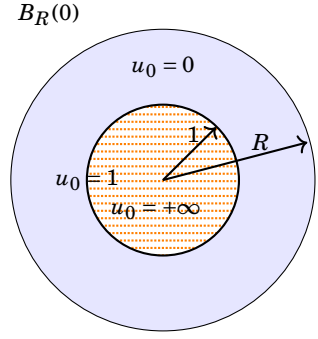


\end{example}

The next example illustrates the case of a generic nonnegative datum $f$ (i.e. the situation of Section \ref{fnon}) and shows that solutions may vanishes on a set of positive measure.  Here for simplicity we assume $\gamma=1$.

\begin{example}\rm
Consider
\begin{equation*}
f(x)=\left\{\begin{array}{lcl}
\frac{N}{\rho} & \text{ if } & 0<|x|\le\rho\,,\\[3mm]
0 & \text{ if } & \rho<|x|<R\,,
\end{array}\right.
\end{equation*}
for some $0<\rho<R$; we define the vector field
\begin{equation*}
z(x)=\left\{\begin{array}{lcl}
\frac{-x}{\rho} & \text{ if } & 0<|x|\le\rho\,,\\[3mm]
\frac{-x}{|x|^N}\rho^{N-1} & \text{ if } & \rho<|x|<R\,,
\end{array}\right.
\end{equation*}
which implies
\begin{equation*}
-\Div z(x)=\left\{\begin{array}{lcl}
\frac{N}{\rho} & \text{ if } & 0<|x|\le \rho\,,\\[3mm]
0 & \text{ if } & \rho<|x|<R\,.
\end{array}\right.
\end{equation*}
Then the solution of \eqref{prob-examples} (with $\gamma=1$) is given by
\begin{equation*}
u(x)=\left\{\begin{array}{lcl}
1 & \text{ if } & 0<|x|\le\rho\,,\\[3mm]
0 & \text{ if } & \rho<|x|<R\,.
\end{array}\right.
\end{equation*}
In fact, it is easy to check that $u \in BV(B_R(0))$, boundary condition holds as well as $(z(x),Du(x))=|Du(x)|$ since
$$Du(x)=(u^+(x)-u^-(x)) \nu(x)\bk\mathcal H^{N-1}\res_{\{|x|=\rho\}} =\frac{-x}{\rho}\mathcal H^{N-1}\res_{\{|x|=\rho\}}\,,$$
so that
\begin{equation*}
(z(x),Du(x))=\left(\frac{-x}{\rho},\frac{-x}{\rho}\right)\mathcal H^{N-1}\res_{\{|x|=\rho\}}=|Du(x)|\,.
\end{equation*}
\end{example}

In this last example we emphasize that  the previous degeneracy is not caused by the smoothness of the datum $f$ and, moreover, we provide a non-uniqueness instance.
\begin{example}\label{ejemplo-no-unicidad}\rm
We now consider, in problem \eqref{prob-examples} with $\gamma=1$, the nonnegative datum
\begin{equation*}
f(x)=\left\{\begin{array}{lcl}
\dfrac{N-1}{|x|^q} & \text{ if } & 0<|x|\le\rho\,,\\[3mm]
0 & \text{ if } & \rho<|x|<R\,,
\end{array}\right.
\end{equation*}
with $0<\rho<R$, $1<q<N$. If $0\le r\le \rho$, the vector field is given by $z(x)=-x/|x|$ and the solution is $u(x)=1/|x|^{q-1} $;   observe that in this case $u$ itself is $BV(B_R (0))$. Taking the vector field $z(x)=-x|x|^{-N}\rho^{N-1}$  in $\rho < r \le R$, it is easy to check that the solution is given by $u(x)=\rho^{1-q}$ for $\rho<|x|<R$. Since the boundary condition holds ($[z,\nu]=(-x/R,x/R) = -1$) as well as the equality $(z(x),Du(x))=|Du(x)|$ as measures, we conclude that $u$ is indeed a solution of \eqref{prob-examples}. Moreover, uniqueness does not hold owing to function
\begin{equation*}
v(x)=\left\{\begin{array}{lcl}
\dfrac{1}{|x|^{q-1}} & \text{ if } & 0<|x|<\rho\,,\\[3mm]
0 & \text{ if } & \rho<|x|<R\,,
\end{array}\right.
\end{equation*}
is also a solution with the same vector field. Indeed, $v\in BV(B_R(0))$, it is straightforward that boundary condition is satisfied and $(z(x),Dv(x))=|Dv(x)|$ holds due to
$$\begin{aligned}
Dv(x)\res_{\{|x|=\rho\}}=(v^+(x)-v^-(x))\nu(x)\mathcal H^{N-1}\res_{\{|x|=\rho\}} &= \left(\frac{1}{\rho^{q-1}}\right)\left(\frac{-x}{\rho}\right) \mathcal H^{N-1}\res_{\{|x|=\rho\}}\\
&=\frac{-x}{\rho^q}\mathcal H^{N-1}\res_{\{|x|=\rho\}}\,,
\end{aligned}$$
and so
\begin{equation*}
(z(x),Dv(x))\res_{\{|x|=\rho\}}=\frac{1}{\rho^{q-1}}\mathcal H^{N-1}\res_{\{|x|=\rho\}} =|Dv(x)|\res_{\{|x|=\rho\}}\,.
\end{equation*}

\end{example}

\section*{Ackonwledgement}
The fourth author is supported by the Spanish Ministerio de Ciencia, Innovaci\'on y Universidades and FEDER, under project PGC2018--094775--B--I00.

\end{document}